\documentclass[a4paper,oneside,11pt]{article}

\usepackage{amsmath,amsfonts,amscd,amssymb,amsthm}
\usepackage{longtable,geometry}
\usepackage[english]{babel}
\usepackage[utf8]{inputenc}
\usepackage[active]{srcltx}
\usepackage[T1]{fontenc}
\usepackage{graphicx}
\usepackage{pstricks}
\usepackage{bbm}
\usepackage{enumitem}
\usepackage{color}
\usepackage{MnSymbol}
\usepackage{stmaryrd}
\usepackage{nicefrac}
\usepackage{calrsfs}
\usepackage{mathtools} 
\usepackage{mdwlist}   
\usepackage{epstopdf} 

\newtheorem{theorem}{Theorem}[section]

\newtheorem{claim}{Claim}
\newtheorem*{claim*}{Claim}

\newtheorem{corollary}[theorem]{Corollary}
\newtheorem{lemma}[theorem]{Lemma}
\newtheorem{proposition}[theorem]{Proposition}

\newtheorem{remark}[theorem]{Remark}

\newcommand{\be}[1]{\begin{equation}\label{#1}}
\newcommand{\ee}{\end{equation}}
\numberwithin{equation}{section}

\newcommand{\ba}[1]{\begin{align}\label{#1}}
\newcommand{\ea}{\end{align}}
\numberwithin{equation}{section}

\newcommand{\ben}{\begin{equation*}}
\newcommand{\een}{\end{equation*}}
\numberwithin{equation}{section}

\renewenvironment{proof}[1][\relax]
  {\paragraph{Proof\ifx#1\relax\else~of #1\fi}}%
  {~\hfill$\square$\par\bigskip}

\newcommand{\calC}{\mathcal{C}}
\newcommand{\calD}{\mathcal{D}}
\newcommand{\calE}{\mathcal{E}}
\newcommand{\calF}{\mathcal{F}}
\newcommand{\calG}{\mathcal{G}}
\newcommand{\calH}{\mathcal{H}}

\newcommand{\calR}{\mathcal{R}}

\newcommand{\bbN}{\mathbb{N}}

\newcommand{\bbR}{\mathbb{R}}

\newcommand{\bbZ}{\mathbb{Z}}

\newcommand{\ga}{\gamma}
\newcommand{\de}{\delta}
\newcommand{\ep}{\varepsilon}
\newcommand{\eps}{\ep}
\newcommand{\om}{\omega}
\newcommand{\Om}{\Omega}

\newcommand{\rk}[1]{\bgroup\color{red}%
  \par\medskip\hrule\smallskip%
  \noindent\textbf{#1}%
  \par\smallskip\hrule\medskip\egroup}

\title{The phase transitions of the planar random-cluster and Potts models with $q \geq 1$ are sharp}
\author{Hugo Duminil-Copin and Ioan Manolescu}
\date{\today}

\newcommand{\Lat}{\calG}
\newcommand{\bfS}{\mathsf{S}}
\newcommand{\bfR}{\mathsf{R}}
\newcommand{\Ball}{\Lambda}
\newcommand{\pd}{\partial}

\newcommand\cond{\, | \,}

\newcommand\lra{\leftrightarrow}
\newcommand\xlra{\xleftrightarrow}

\DeclarePairedDelimiter{\floor}{\lfloor}{\rfloor}

\newcommand\cf{c_\mathbf{f}}  
\usepackage{psfrag}
\def\mik{1}
\newcommand\cpsfrag[2]{\ifnum\mik=1\psfrag{#1}{#2}\fi}

\newcommand{\Latk}{\tilde{\Lat}}  

\begin{document}

\maketitle

\begin{abstract}
  We prove that random-cluster models with $q \geq 1$ on a variety of planar lattices have a sharp phase transition, 
  that is that there exists some parameter $p_c$ below which the model exhibits exponential decay 
  and above which there exists a.s. an infinite cluster. 
  The result may be extended to the Potts model via the Edwards-Sokal coupling. 
  
  Our method is based on sharp threshold techniques and certain symmetries of the lattice; 
  in particular it makes no use of self-duality. 
  Part of the argument is not restricted to planar models and may 
  be of some interest for the understanding of random-cluster and Potts models in higher dimensions.
 
  Due to its nature, this strategy could be useful in studying other planar models 
  satisfying the FKG lattice condition and some additional differential inequalities.  
  \end{abstract}

\section{Introduction}

\paragraph{Main statement.}
The random-cluster model (or FK percolation) was introduced by Fortuin and Kasteleyn in 1969 
as a class of models satisfying specific series and parallel laws. 
It is related to many other models, including the $q$-state Potts models ($q = 2$ being the particular case of the Ising model). 
In addition to this, the random-cluster model exhibits a variety of interesting features, 
many of which are still not fully understood. 

Consider a finite graph $G = (V_G,E_G)$. 
The random-cluster measure with edge-weight $p\in [0,1]$ and cluster-weight $q > 0$ on $G$ 
is a measure $\phi_{p,q,G}$ on configurations $\omega\in\{0,1\}^{E_G}$. An edge is said to be {\em open} (in $\omega$) if $\omega(e)=1$, otherwise it is {\em closed}. The configuration  $\omega$ can be seen as a subgraph of $G$ with vertex set $V_G$ and edge-set $\{e\in E_G:\omega(e)=1\}$. 
A {\em cluster} is a connected component of $\omega$. 
Let $o(\omega)$, $c(\omega)$ and $k(\omega)$ denote the number of open edges, closed edges and clusters  in $\omega$ respectively. 
The probability of a configuration is then equal to
\begin{equation*}
\phi_{p,q,G}(\omega)=\frac{p^{o(\omega)}(1-p)^{c(\omega)}q^{k(\omega)}}{Z(p,q,G)},
\end{equation*}
where $Z(p,q,G)$ is a normalizing constant called the partition function. 

Consider a connected planar locally-finite doubly periodic graph $\Lat$, 
i.e. a graph which is invariant under the action of some lattice $\Lambda\simeq\bbZ\oplus\bbZ$.  
The model can be extended to $\Lat$ by taking limits of measures 
on finite graphs $G_n$ tending to $\Lat$ 
(with certain boundary conditions, see Section~\ref{sec:basic_prop} for details).
We call such limits \emph{infinite-volume} measures. 
As discussed later, for any pair of parameters $p \in [0,1]$ and $q \geq 1$, 
at least one infinite-volume measure exists, but it is not necessarily unique.  
For $q\ge 1$, the infinite-volume model exhibits a phase transition at some critical parameter $p_c(q)$ (depending on the lattice). 
The aim of the present paper is to give a  proof of the sharpness of this phase transition.

\begin{theorem}\label{thm:main}
  Fix $q\ge1$. Let $\Lat$ be a planar locally-finite doubly periodic connected graph 
  invariant under reflection with respect to the line $\{(0,y), y \in \bbR\}$
  and rotation by some angle $\theta \in (0,\pi)$ around $0$. 
  There exists $p_c=p_c(\Lat) \in [0,1]$ such that
  \begin{itemize}[nolistsep,noitemsep]
  \item for $p < p_c$, there exists $c=c(p,\Lat)>0$ such that for any $x,y\in\Lat$,
    \begin{align}
      \phi_{p,q}[x\text{ and }y\text{ are connected by a path of open edges}]\le \exp(-c|x-y|), \label{eq:exp_decay0}
    \end{align}
  \item for $p > p_c$, there exists a.s. an infinite open cluster under $\phi_{p,q}$,
  \end{itemize}
  where $\phi_{p,q}$ is the unique infinite-volume random-cluster measure on $\Lat$ 
  with edge-weight $p$ and cluster-weight $q$.
\end{theorem}
\begin{remark}
	The fact that, for $p \neq p_c$, there exists a unique infinite-volume measure with edge-weight $p$
	may easily be shown by adapting \cite[Thm. 6.17]{Gri06}.
\end{remark}

The sharpness of the phase transition was proved in arbitrary dimension 
for percolation in \cite{AizBar87,Men86} 
and for the Ising model in \cite{AizBarFer87}. 
For planar random-cluster models with arbitrary cluster-weight $q\ge1$, 
the sharpness had been previously derived only in the case of the 
square, triangular and hexagonal lattices, see \cite{BefDum12}. 
A similar result is proved for so-called isoradial graphs in \cite{DumMan13}. 
It may be worth mentioning that, contrary to the present work, \cite{BefDum12} and \cite{DumMan13} 
are both based on integrability properties of the model.


The exponential decay of the two-point function is key to the study of the subcritical phase. 
It implies properties such as 
exponential decay of the cluster-size, finite susceptibility, Ornstein-Zernike estimates and mixing properties, to mention but a few.
We do not go into details here, but rather refer the reader to the monographs \cite{Gri99a,Gri06} for further reading. 

Our method is based on the \emph{sharp threshold} property and on certain symmetries of the lattice.
A corollary of our results is that self-dual models are critical. 
\begin{corollary}
  The critical parameters $p_c(q)$ of the square, triangular and hexagonal lattices satisfy 
  \begin{align*}
  &\text{on the square lattice: } p_c(q)=\sqrt q/(1+\sqrt q),\\
  &\text{on the triangular lattice: }  p_c(q) \text{ is the unique solution $p$ in $[0,1]$ of }p^3+3p^2(1-p)=q(1-p)^3,\\
  &\text{on the hexagonal lattice: } p_c(q) \text{ is the unique solution $p$ in $[0,1]$ of }p^3-3qp(1-p)^2=q^2(1-p)^3.
  \end{align*}
\end{corollary}
The model on the square lattice with the above parameter is indeed self-dual;
the ones on the triangular and hexagonal lattices are not \emph{per se}. 
They are dual to each other, but also related through the star--triangle transformation (see \cite[Sec. 6.6]{Gri06}).

As mentioned above, the previous corollary was obtained in \cite{BefDum12}. Nevertheless, the present method has the advantage of using self-duality for the identification of the critical point only, and not for the proof of sharpness (in \cite{BefDum12}, the self-duality is used in the proof of a Russo-Seymour-Welsh type estimate leading to the sharpness of the phase transition).

\paragraph{Extensions of Theorem~\ref{thm:main}}
We discuss several (potential) generalisations of the previous theorem.

First, the biperiodic graph $\Lat=(V_\Lat,E_\Lat)$ may be replaced by a weighted biperiodic graph $(\Lat,J)$, 
where $J$ is a family of strictly positive weights on edges. 
For any subgraph $G=(V_G,E_G)$ of $\Lat$ and $\beta \geq 0$, we define
\begin{equation}\label{eq:def J}
	\phi_{\beta,q,G,J}(\omega)=\frac{\Big(\prod_{e\in E_G}(e^{\beta J_e}-1)^{\omega(e)}\Big)\cdot q^{k(\omega)}}{Z(\beta,q,G,J)},
\end{equation}
where $Z(\beta,q,G,J)$ is a normalizing constant.
One may easily see that in the case of $J_e=J$ for any $e\in E_G$, we obtain the previous definition with $p=1-e^{-J\beta}.$
As before, infinite-volume measures may be defined on $\Lat$ by taking limits. 

\begin{theorem}\label{thm:main_inhom}
  Fix $q\ge1$. Let $\Lat$ be a planar locally-finite doubly periodic connected weighted graph 
  invariant under reflection with respect to the line $\{(0,y), y \in \bbR\}$
  and rotation by some angle $\theta \in (0,\pi)$ around $0$. 
  There exists $\beta_c=\beta_c(\Lat,J) \geq 0$ such that
  \begin{itemize}[nolistsep,noitemsep]
  \item for $\beta < \beta_c$, there exists $c=c(\beta,\Lat,J)>0$ such that for any $x,y\in\Lat$,
    $$\phi_{\beta,q,J}[x\text{ and }y\text{ are connected by a path of open edges}]\le \exp(-c|x-y|),$$
  \item for $\beta > \beta_c$, there exists a.s. an infinite open cluster under $\phi_{\beta,q,J}$,
  \end{itemize}
 where $\phi_{\beta,q,J}$ is the unique infinite-volume random-cluster measure 
  on $\Lat$ with parameters $q$ and $\beta$. 
\end{theorem}
The proof of this theorem follows {\em exactly} the same lines as the one of Theorem~\ref{thm:main} 
except that the notation becomes heavier. Thus we will only focus on Theorem~\ref{thm:main}.

A second potential extension is to planar random-cluster models with finite range interactions. 
Consider a planar graph $\Lat=(V_\Lat,E_\Lat)$ with the properties of Theorem~\ref{thm:main}. 
For some $R \geq 1$ define a modified graph $\Latk =(V_\Lat,E_{\Latk})$, 
with same vertex set as $\Lat$ but with $(u, v) \in E_{\Latk}$
if the graph distance between $u$ and $v$ in $\Lat$ is less than or equal to $R$. 
(For $R= 1$, $\Latk = \Lat$.)

We believe that our methods may be modified to prove Theorem~\ref{thm:main} 
(and its inhomogeneous version Theorem~\ref{thm:main_inhom}) for $\Latk$. 
In particular we expect that Theorem~\ref{thm:main} also applies to the random-cluster model on slabs, 
i.e. on the graphs of the form $\Lat \times \{0, \dots, R\}^d$ with $d, R \geq 1$. 
We discuss this further in a forthcoming article. 

A final potential extension is to models other than the random-cluster model. 
Our arguments are somewhat generic, 
and one can try to use them for models similar to those studied here. 
More precisely, to obtain our result, we only need the model to satisfy the conditions listed in Section~\ref{sec:conditions}.
We discuss this point further in Section~\ref{sec:conditions}, when the appropriate notation is in place.

\paragraph{Consequences for the Potts model.}

Fix some finite weighted graph $\big(G,J\big)$, 
where $J = (J_e)_{e \in E_G}$ is a family of positive real numbers.
Also fix a set of parameters $\beta \geq 0$ and $q \in \bbN$ with $q \geq 2$. 
The Potts model on $G$ with $q$ states and inverse temperature $\beta$ 
is a probability measure $\mu_{\beta, q,G,J}$ on $\{1, \dots, q\}^{V_G}$,
for which the weight of a configuration $\sigma$ is given by 
$$ \mu_{\beta, q, G,J}(\sigma) = \frac{e^{-\beta H_{q,G,J}(\sigma)}}{Z_{\beta, q, G,J}^{\textrm{Potts}}},$$
where 
$$H_{q,G,J}(\sigma)=-\sum_{e=(x,y)\in E_G}J_e{\bf 1}_{\sigma_x=\sigma_y}$$
and $Z_{\beta, q, G,J}^{\textrm{Potts}}$ is a normalizing constant.
The sum in the second equation is taken over all unordered pairs of neighbours $x,y$. 

A well-known coupling (sometimes called the Edwards-Sokal coupling) links the Potts and random-cluster models. 
We only briefly describe how to obtain the former from the latter. 
For details see \cite[Thm 4.91]{Gri06}.

Choose a random-cluster configuration $\om$ according to $\phi_{\beta,q,G,J}$, where $\phi_{\beta,q,G,J}$ is defined as in~\eqref{eq:def J}.
Assign to each cluster of $\omega$ a state (or colour) chosen uniformly in $\{1, \dots, q\}$, independently for different clusters.
This generates a random configuration $\sigma \in \{1, \dots, q\}^{V_G}$.
(Note the two sources of randomness used in generating $\sigma$: 
the randomness in the choice of $\om$ and that in the colouring of the clusters of $\om$.)
Then $\sigma$ follows the Potts measure $\mu_{\beta, q, G,J}$.

Consider now a planar locally-finite doubly periodic weighted graph $(\Lat,J)$. 
As for the random-cluster, infinite-volume Potts measures may be defined. 
The phase transition in this case is decided by the existence of long-range correlations. 
In particular, if $\beta_c$ is the critical parameter, then 
\begin{itemize}[noitemsep,nolistsep]
\item for $\beta < \beta_c$, there exists a unique infinite-volume measure (long-range correlations vanish), 
\item for $\beta > \beta_c$, there exist multiple infinite-volume measures (long-range correlations exist).
\end{itemize}
It follows trivially from the above coupling that 
$$\mu_{\beta, q, G,J}(\sigma_x =  \sigma_y) = \frac1{q} + \frac{q-1}{q}\phi_{\beta, q, G,J}(x\text{ and }y\text{ are connected by a path of open edges}),$$
hence the phase transition of the Potts model can be linked to that of the associated random-cluster model. 
In particular, when $J_e = J$ for all $e \in E_{\Lat}$, 
$\beta_c(q) = -\frac{1}{J} \log ( 1- p_c(q))$. 

Our main result may be translated as follows. 
\begin{theorem}
  Fix $q\ge2$. Let $\Lat$ be a planar locally-finite doubly periodic connected weighted graph 
  invariant under reflection with respect to the line $\{(0,y), y \in \bbR\}$
  and rotation by some angle $\theta \in (0,\pi)$ around $0$. 
  There exists $\beta_c = \beta_c(\Lat,J) \geq 0$ such that, 
  \begin{itemize}[nolistsep,noitemsep]
  \item for $\beta < \beta_c$, there exists a unique infinite-volume Potts measure $\mu_{\beta,q,J}$ 
  	with parameters $\beta$ and $q$ on $(\Lat, J)$. 
    Moreover there exists $c=c(\beta,\Lat)>0$ such that for any $x,y\in\Lat$,
    \begin{align*}
      \mu_{\beta,q,J}[\sigma_x = \sigma_y] - \frac1{q}\le \exp(-c|x-y|), 
    \end{align*}
  \item for $\beta > \beta_c$, there exist multiple infinite-volume Potts measures with parameters $\beta$ and~$q$ on~$(\Lat, J)$. 
  \end{itemize}
\end{theorem}

\paragraph{Strategy of the proof.} 

Let $\phi_{p,q}^0$ be the infinite-volume measure on $\Lat$ with free boundary conditions (see the next section for a precise definition). 
It is obtained as the limit of random-cluster measures $\phi_{p,q,G_n}$ 
on finite subgraphs $G_n$ of $\Lat$ that tend increasingly to $\Lat$.
Define
\begin{align*}
  p_c & := \inf\big\{p\in(0,1)~:~\phi_{p,q}^0(x\text{ is connected by a path of open edges to infinity})>0\big\}\\
  \tilde p_c & := \sup \big\{ p \in(0,1) ~:~ \lim_{n\rightarrow \infty}-\tfrac1n
  \log\big[\phi_{p,q}^0(0\text{ and }\partial\Lambda_n\text{ are connected by a path of open edges})\big]>0\big\}.
\end{align*}
Note that $\tilde p_c\le p_c$. 
We wish to prove that $p_c=\tilde p_c$ (this is simply another way of stating the main result), 
and we therefore focus on the inequality $\tilde p_c\ge p_c$. 
The proof of the latter is based on the study of probabilities of crossing rectangles.
For the sake of simplicity, 
let us restrict our attention in this introduction to rectangles of width $2n$ and height $n$,
i.e. translates of $[0,2n] \times [0,n]$.
A rectangle is {\em crossed  horizontally (vertically)} if it contains a path of open edges
going from its left side to its right side (respectively from the bottom side to the top side). 
The strategy follows three main steps:

\begin{description}
\item[{\em Step 1.}] 
  {\em We first prove that for any $p > \tilde p_c$, the probability of crossing vertically (i.e. in the ``easy direction'') 
    a rectangle of size $2n \times n$ is bounded away from 0 uniformly in $n$}. 
  
  We show this by proving that for any $0<\eps<p$, if the $\phi_{p,q}^0$-probability 
  of crossing vertically a rectangle of size $2n \times n$ 
  drops below a certain benchmark (even for a single value of $n$), 
  then the $\phi_{p-\eps,q,\calG}$-probability that two points are connected by an open path decays exponentially 
  fast  (see Proposition~\ref{prop:easy} for the precise statement). 
  A similar (but stronger) statement was proved by Kesten for percolation \cite{Kes82}. 
  He proved that, given a percolation measure, if the probability of crossing the rectangle vertically is too small, 
  then exponential decay follows for that measure. 
  The difference with our result is that, in the case of percolation, one does not need to alter the parameter of the measure (see Remark~\ref{rem:10} for more details). 
  
  We highlight the fact that this part of the proof is not specific to the planar case. 
  
\item[{\em Step 2.}] 
  {\em Using the first step, we show that for any $p > \tilde p_c$, 
  the probability of crossing horizontally (i.e. in the ``hard direction'') 
  a rectangle of size $2n\times n$ is bounded away from 0 uniformly in $n$.}
  
  This step is the most difficult. It corresponds to proving a ``Russo-Seymour-Welsh'' (RSW) type result: 
  if crossing probabilities in the easy direction are bounded away from 0, then it is the same in the hard direction. 
  Such results were first proved in the context of Bernoulli percolation on the square lattice \cite{Rus78,SeyWel78}. 
  Similar statements have been recently obtained for the Ising model \cite{DumHonNol11,CheDumHon12} 
  and the random-cluster models with cluster-weight $1 \leq q \leq 4$ \cite{DumSidTas13}, but only for the square lattice. 
  These results usually represent the first step towards a deep understanding on the critical phase.

  In the present paper we prove a weaker statement than these RSW results: 
  we show that, if crossing probabilities in the easy direction are bounded away from $0$ for some edge-weight $p$, 
  then it is the same in the hard direction for any $p'>p$. 
  As in the first step, the difference with previous results 
  is that we need to increase the edge-weight to obtain the desired conclusion.

\item[{\em Step 3.}] 
  {\em We show that if $p < p' <p_c$ are such that the $\phi_{p,q}^0$-probability 
  of crossing horizontally a rectangle of size $2n\times n$ is bounded away from 0 uniformly in $n$, 
  then the $\phi_{p',q}^0$-probability of these events tends to $1$ as $n$ tends to $\infty$.}
 
   This step is based on an argument from \cite{GraGri11}
   that combines an influence theorem and a coupling argument to
   obtain a sharp threshold inequality (see Corollary~\ref{cor:gg_applied}). 
\end{description} 

Observe that these steps combine together to give the proof of the theorem. 
Indeed suppose $\tilde p_c < p_c$ and take  $\tilde p_c < p_0 < p_1 <  p_2 < p_c$.
By steps 1 and 2, the probabilities under $\phi_{p_0,q}^0$ of crossing in the hard direction 
rectangles of size $2n \times n$ are bounded away from $0$, uniformly in $n$. 
By step $3$ these crossing probabilities tend to $1$ under $\phi_{p_1,q}^0$. 
As a consequence the probability of a dual crossing 
in the easy direction of a $2n$ by $n$ rectangle tends to $0$. 
But step $1$ also applies to dual measures, hence, for the edge-weight $p_2$, 
the two-point function of the dual model decays exponentially fast. 
This implies via a classical argument that there exists an infinite-cluster in the primal model, 
and this is a contradiction. 

\begin{remark}\label{rem:10}
  The proofs of Steps 1 and 2 require varying the edge-weight $p$. Nevertheless, we expect that this is not indispensable. 
  Bernoulli percolation is an example for which the proofs of Steps 1 and 2 are valid without changing $p$, 
  but the known proofs of this fact rely heavily on independence. 
  In order to tackle more general models (in particular those having long-range dependence), 
  we employ the differential inequality~\eqref{eq:hamming} invoking the Hamming distance, which entails altering $p$. 
  The related differential inequality~\eqref{eq:ggsh1} is used in Step 3. 
  Exploiting them to their full strength is the main novelty of this article.
\end{remark}

\paragraph{Open questions.} 
We end this introduction by mentioning three related open questions. 

The first is to investigate to which other models the methods of this paper may be adapted. 
We discuss this in Section~\ref{sec:conditions}, where we identify specific conditions for such models.

The second is to obtain results similar to Theorem~\ref{thm:main} for lattices in dimensions $d > 2$. 
We believe that some of the techniques presented in this article can be harnessed in more general dimensions 
(we think in particular of Step 1 and inequalities~\eqref{eq:ggsh1} and~\eqref{eq:hamming}). 
Nevertheless, the methods of Steps 2 and 3 are based on certain features of planarity, 
and we are currently unable to extend Theorem~\ref{thm:main} to higher dimension.

Finally we mention a broader direction of research. 
Just as the method of \cite{BefDum12}, our article provides very little information 
on the critical phase of the random-cluster model. 
Recent results (for instance \cite{DumSidTas13,Smi10,CheDumHon12a,Dum13}) 
have illustrated that it is possible extract knowledge of the critical phase of random-cluster models 
from the theory of discrete holomorphic observables.  
But this theory is often based on integrability properties of the model, 
properties which are not true for general random-cluster models on planar locally-finite doubly periodic graphs. Therefore, it is very challenging to understand how to extend our knowledge 
of the critical random-cluster model on the square lattice to more general settings.
A first step towards this goal is to prove that the results of Steps 1 and 2 are valid without changing the edge-parameter.

\paragraph{Organisation of the paper.}
Section~\ref{sec:notation} is dedicated to defining the model 
and explaining the properties needed in the proof of Theorem~\ref{thm:main}.
The next sections follow the steps described above:
in Sections~\ref{sec:proof_RSW} and~\ref{sec:proof_RSW2} 
we prove two finite size criteria for exponential decay (corresponding to Steps~1 and~2)
that we then use in Section~\ref{sec:mainproof} to prove our main theorem (this corresponds to Step~3). 
In Section \ref{sec:extensions} we investigate a possible extension of the result to more general models. 

\section{Notations and basic facts on the model}\label{sec:notation}
\subsection{Graph definitions}

\paragraph{The lattice $\Lat$.}
Fix for the rest of the paper a locally-finite planar connected graph $\Lat = (V_\Lat,E_\Lat)$ embedded in the plane $\bbR^2$ 
(in such a way that edges are straight lines intersecting at their end-points only)
and assume there exist $u$ and $v\in\bbR^2$ non collinear, and $\theta\in(0,\pi)$ 
such that the following maps are graphs automorphisms of the embedded graph $\Lat$:
\begin{itemize}[noitemsep,nolistsep]
\item the translations by vectors $u$ and $v$, 
\item the rotation of angle $\theta$ around 0,
\item the orthogonal reflection with respect to the vertical line $\{(0,y),y\in\bbR\}$. 
\end{itemize}
It may be seen that, since $\Lat$ is required to be locally finite, 
there are only two possible values for $\theta$, namely $\frac{\pi}{3}$ and $\frac{\pi}{2}$.
The triangular lattice is an example corresponding to the first case, while the square lattice corresponds to the second
(obviously, other examples may be given in both cases). 
For simplicity, we will only treat the case $\theta=\frac\pi2$ in the following; 
the results also hold in the case $\theta = \frac{\pi}{3}$,  with some standard adjustments of the proofs. 
It may be shown that, if we allow some rescaling, we may consider the lattice to be invariant by
\begin{itemize}[nolistsep,noitemsep]
\item the translations by $(1,0)$ and $(0,1)$, 
\item rotation by $\frac{\pi }{2}$ around the origin,
\item the orthogonal symmetry with respect to the vertical line $\{(0,y), y \in \bbR\}$.
\end{itemize}
In the rest of this article, the graph $\Lat$ will be referred to  as \emph{the lattice}. 
Two vertices $x$ and $y$ of $V_\Lat$ are said to be neighbours if $(x,y)\in E_\Lat$. We then write $x\sim y$.

The graph $G=(V_G,E_G)$ will always denote a finite subgraph of $\Lat$, i.e.  
$E_G$ is a finite subset of $E_\Lat$ and $V_G$ is the set of end-points of $E_G$. 
We denote by $\pd G$ the boundary of $G$, i.e. $$\pd G=\{x\in V_G:\exists y\notin V_G\text{ with }x\sim y\}.$$

For $a < b$ and $c <d$, let $R = [a,b] \times [c,d]$ be the subgraph of $\Lat$ induced by the vertices of $V_\Lat$ in $[a,b]\times[c,d]$. 
This type of graph will be called a \emph{rectangle}. For $n \geq 0$, let $\Ball_n = [-n,n]^2$. 

\paragraph{Dual lattice and dual graphs.}
Let $\Lat^*$ be the dual lattice of $\Lat$, obtained by placing a vertex in each face of $\Lat$ 
and joining two vertices of $\Lat^*$ if the corresponding faces of $\Lat$ are adjacent. 
Note that $\Lat^*$ enjoys the same symmetries as $\Lat$. 
For $e\in E_\Lat$, set $e^*$ for the edge of $\Lat^*$ intersecting $e$.  
For a finite graph $G$, define $G^*$ to be the graph with edge-set $E_{G^*}:=\{e^*,e\in E_G\}$, 
and vertex-set $V_{G^*}$ given by the end-points of edges in $E_{G^*}$.

\paragraph{The space of configurations.} 
Let $G=(V_G,E_G)$ be a subgraph of $\calG$. We will always work with elements $\om$ of $\Om = \{0,1\}^{E_G}$, called  \emph{configurations}. 
Edges $e$ with $\om(e) = 1$ are called \emph{open} (in $\omega$), while others are \emph{closed} (in $\omega$).
As mentioned above, $\om$ can be seen as a subgraph of $G$ whose vertex-set is $V_G$ and edge-set is $\{e\in E_G:\omega(e)=1\}$. 

A path on $G$ is a sequence of vertices $u_0, \dots, u_n\in V_G$ 
with $(u_i,u_{i+1}) \in E_G$ for $i = 0, \dots, n-1$. 
It is called {\em open} if  $(u_i,u_{i+1})$ is open in $\omega$ for every $i$. 
Two vertices $a$ and $b$ are said to be \emph{connected} (in $\om$ on $G$),
if there exists an open path connecting them. 
The event that $a$ and $b$ are connected is denoted by  $a \xlra{\om, G} b$ 
(or simply $a\xlra{G}b$ or even $a \xlra{} b$ when no confusion is possible).
Two sets $A$ and $B$ are connected (denoted $A\xlra{} B$) 
if there exists a pair of connected vertices $(a,b) \in A\times B$.  
A maximal set of connected vertices is called a \emph{cluster}. 

When $G = [a,b]\times [c,d]$ is a rectangle and $A=\{a\}\times[c,d]$ and $B=\{b\}\times[c,d]$ 
(respectively $A=[a,b]\times\{c\}$ and $B=[a,b]\times\{d\}$),
the event $A\xlra{\om, G} B$ is also denoted $\calC_h([a,b]\times[c,d])$ (respectively $\calC_v([a,b]\times[c,d])$)
and if it occurs we say that $G$ is crossed horizontally (respectively vertically). 
An open path from $A$ to $B$ is called a horizontal crossing (respectively vertical crossing).
When $a=0$ and $c=0$, we simply write $\calC_h(b,d)$ and $\calC_v(b,d)$ for the events above. 
  When $b-a > d-c$, horizontal crossings are called crossings in the hard direction, 
  while vertical ones are crossings in the easy direction. 
  The terms are exchanged when $b-a < d-c$.

To each configuration $\om\in \Om$ is associated a dual configuration $\om^*$ on $G^*$ defined by 
$\om^*(e^*)=1-\om(e)$. A dual-path on $G^*$ is a sequence of vertices $u_0, \dots, u_n\in V_{G^*}$ 
with $(u_i,u_{i+1}) \in E_{G^*}$ for $i = 0, \dots, n-1$. 
It is called {\em dual-open} if  $\om^*(u_i,u_{i+1})=1$ for all $i$. 
Two dual-vertices $u$ and $v$ are said to be \emph{dual-connected} 
(written $u \xlra{\om^*,G^*} v$ or simply $u \xlra{*} v$ when no confusion is possible)
if there is a dual-open path connecting them.
A maximal set of connected dual-vertices is called a \emph{dual-cluster}. 
The definitions of crossings extend to dual configurations in the obvious way. 

\subsection{Basic properties of the random-cluster model}\label{sec:basic_prop}

For more details and proofs we direct the reader to \cite{Gri06} or \cite{Dum13}.

\paragraph{Boundary conditions.}
Let $G = (V_G,E_G)$ be a finite subgraph of~$\Lat$. 
A \emph{boundary condition}~$\xi$ is a partition of~$\partial G$. 
We denote by~$\om^\xi$ the graph obtained from the configuration~$\om$ 
by identifying (or \emph{wiring}) the vertices in~$\partial G$ that belong to the same element of the partition~$\xi$.
Boundary conditions should be understood as encoding how vertices are connected outside~$G$. 
The probability measure $\phi^{\xi}_{p,q,G}$ of the random-cluster model on~$G$ 
with parameters $p \in [0,1]$, $q \geq 0$ and boundary condition~$\xi$ 
is defined on~$\Om$ by
\begin{equation}
  \label{probconf}
  \phi_{p,q,G}^{\xi} (\omega) :=
  \frac {p^{o(\omega)}(1-p)^{c(\omega)}q^{k(\omega^\xi)}}
  {Z^{\xi}(p,q,G)},
\end{equation}
where $Z^{\xi}(p,q,G)$ is a normalizing constant referred to as the \emph{partition function}. 
Above, $o(\omega)$, $c(\omega)$ and $k(\omega^\xi)$ correspond to the number of open and closed edges of $\om$, 
and the number of clusters of $\om^\xi$.

Two specific boundary conditions are particularly important. 
The free boundary condition, denoted $0$, correspond to the partition composed of singletons only (no wiring between boundary vertices).
The wired boundary condition, denoted $1$, correspond to the partition $\{\partial G\}$ (all vertices are wired together). 
In addition to these two, we will sometimes consider boundary conditions induced by a configuration $\xi$ outside $G$:
two vertices are wired together if there exists a path between them in $\xi$. 
We will identify $\xi$ with the induced boundary condition and simply write $\phi_{p,q,G}^\xi$ for the corresponding measure. 


\paragraph{Domain Markov property.} 

Let $G\subset F$ be two finite subgraphs of $\Lat$. 
A configuration $\omega$ on $F$ may be viewed as a configuration on $G$ by taking its restriction $\omega_{|G}$ to edges of $G$. 
The restriction of the configuration $\omega$ to edges of $F \setminus G$ induces boundary conditions on $G$ as explained below. 
The domain Markov property states that for any $p,q$, any boundary condition $\xi$ on $F$ and any $\psi\in \{0,1\}^{E_F\setminus E_G}$,
\begin{equation}\label{domain Markov}
  \phi_{p,q,F}^\xi(\omega_{|G}=\cdot\,|\omega(e)=\psi(e),e\in E_F\setminus E_G)=\phi_{p,q,G}^{\psi^\xi}(\cdot),
\end{equation}
where $\psi^\xi$ is the partition induced by the equivalence relation $x\calR y$ if $x$ and $y$ are connected in $\psi^\xi$.

The domain Markov property implies the following finite-energy property. 
For any $\ep>0$, the conditional probability for an edge to be open, 
knowing the states of all the other edges, is bounded away from 
$0$ and $1$ uniformly in $p\in[\ep,1-\ep]$ and in the state of other edges. 
This property extends to finite sets of edges (with a constant which gets worse and worse as the cardinality of the set increases).

\paragraph{Stochastic ordering for $q\ge1$.} 
For any $G$, the set $\{0,1\}^{E_G}$ has a natural partial order. 
An event $A$ is increasing if for any $\omega\le \omega'$, $\omega\in A$ implies $\omega'\in A$. 
The random-cluster model satisfies the following properties:
\begin{enumerate}
\item (FKG inequality) Fix $p\in[0,1]$, $q\ge 1$ and some boundary condition $\xi$. 
  Let $A$ and $B$ two increasing events, then $\phi_{p,q,G}^\xi(A\cap B)\ge\phi^\xi_{p,q,G}(A)\phi^\xi_{p,q,G}(B)$.
\item (comparison between boundary conditions) Fix $p\in[0,1]$, $q\ge 1$ and $\xi$ and $\psi$ two boundary conditions. 
  Assume that $\xi\le \psi$, meaning that the partition $\psi$ is coarser than $\xi$ (there are more wirings in $\psi$ than in $\xi$), 
  then for any increasing event $A$, $\phi_{p,q,G}^\xi(A) \le \phi_{p,q,G}^\psi(A)$.
\item (comparison between different edge parameters) Fix $p_1\le p_2$, $q\ge 1$ and some boundary condition $\xi$. 
  Then for any increasing event $A$, $\phi_{p_1,q,G}^\xi(A)\le \phi_{p_2,q,G}^\xi(A)$.
\end{enumerate}

\paragraph{Infinite-volume measures for $q\ge1$.}
We will consider measures on infinite-volume configurations, i.e. on $\{0,1\}^{E_\Lat}$.
Recall that for any finite subgraph $G$ of $\Lat$, a configuration $\om \in \{0,1\}^{E_\Lat}$ 
induces a boundary condition on $G$ that we will exceptionally write in this paragraph $\chi(\om)$.
Under $\chi(\om)$, two vertices $x,y \in \partial G$ are wired if and only if they are connected in $\om$ on $\Lat \setminus G$. 
An infinite-volume random-cluster measure on $\Lat$ with parameters $p$ and $q$ 
is a measure $\phi_{p,q}$ on $\{0,1\}^{E_\Lat}$ with the property that, for all finite subgraphs $G$ of $\Lat$, 
\begin{align}\label{eq:def infinite}
  \phi_{p,q} (\om _{|G}=\cdot\,  \cond \chi(\om) = \xi) = \phi_{p,q,G}^\xi(\cdot),
\end{align}
for all boundary conditions $\xi$ for which the conditioning is not degenerate.

The properties of the previous paragraph extend to infinite-volume measures by~\eqref{eq:def infinite}.

One may prove that for any pair of parameters $(p,q)$, there exists at least one such measure. 
When $q\ge1$, one may for instance take the limit of measures with wired (resp. free) boundary conditions on $\Lambda_n$. 
The measure obtained in the limit is called the infinite-volume measure with wired (resp. free) 
boundary conditions and is denoted by $\phi_{p,q}^1$ (resp. $\phi_{p,q}^0$).

In general there is no reason that, for a given pair of parameters $(p,q)$, there is a unique infinite-volume measure. 
Nevertheless, for $q\geq 1$, the set $\mathcal{D}_q$ of values of $p$ for which there exist at least two distinct infinite-volume measures 
is at most countable, see \cite[Theorem (4.60)]{Gri06}.  
This property can be combined with the stochastic ordering between different edge-weights 
to show the existence of a \emph{critical point} $p_c\in[0,1]$ 
such that:
\begin{itemize}[noitemsep,nolistsep]
\item for any infinite-volume measure with $p<p_c$, 
  there is almost surely no infinite cluster, 
\item for any infinite-volume measure with $p>p_c$, 
  there is almost surely an infinite cluster. 
\end{itemize}
When the planar, locally finite, doubly-periodic graph is non-degenerate, 
$p_c$ can be proved to be different from 0 and 1 using a variant of the classical Peierls argument.

While the above is true also for lattices in higher dimensions, 
for planar lattices such as $\Lat$ an additional argument shows that $\mathcal{D}_q \subseteq \{p_c\}$ (in fact in any dimension one has $\mathcal D_q\subseteq[p_c,1]$, see \cite[Theorem 5.16]{Gri06}).
See the discussion following Remark~\ref{rem:dual_exp_decay0} for details.  

\paragraph{Planar duality.} 

Let $G$ be a finite graph and $\xi\in\{0,1\}^{E_\Lat\setminus E_G}$. 
If $\om$ is distributed according to $\phi_{p,q,G}^\xi$, 
the configuration $\om^*$ is also distributed as a random-cluster configuration on $G^*$ with different parameters. 
More precisely, we find that
$$\phi_{p,q,G}^\xi(\omega)=\phi_{p^*,q^*,G^*}^{\xi^*}(\omega^*),$$
where 
$$\frac{pp^*}{(1-p)(1-p^*)} =q\quad\text{and}\quad q^*=q$$
and $\xi^*$ is the boundary condition on $\partial G^*$ 
induced by the dual-configuration $\xi^*\in\{0,1\}^{E_{\Lat^*}\setminus E_{G^*}}$.
For instance, dual measures extend to the whole of $\Lat^*$ and, if $\omega$ follows $\phi_{p,q, \Lat}^1$ (respectively $\phi_{p,q, \Lat}^0$), 
then $\omega^*$ is distributed as $\phi_{p^*,q, \Lat^*}^0$ (respectively $\phi_{p^*,q, \Lat^*}^1$).

\begin{remark}\label{rem:dual_exp_decay0}
	As a consequence of Theorem~\ref{thm:main}, for $q \geq 1$ and $p > p_c$, there exists $c = c(p, q) > 0$ such that
	\begin{align}\label{eq:dual_exp_decay0}
    	\phi_{p, q}(u\stackrel{*}{\longleftrightarrow}v)\le \exp(-c|u-v|), \quad \text{ for all } u,v \in V_{\Lat^*}.
	\end{align}
\end{remark}
Indeed, an adaptation of Zhang's argument (as that of \cite[Thm 6.17]{Gri06}) shows that,
for any values of $q \geq 1$ and $p \notin \calD_q$, it is impossible to have with positive probability 
infinite clusters in both $\omega$ and $\omega^*$. 
Thus, if $p > p_c$, there is no infinite cluster in $\omega^*$, 
and Theorem~\ref{thm:main} applied to the dual random-cluster model implies~\eqref{eq:dual_exp_decay0}.


\paragraph{Differential inequalities.}

The two following theorems are essential to our study. 
The first is a direct adaptation of the more general statement of Graham and Grimmett \cite[Thm. 5.3]{GraGri11}.

\begin{theorem}[\cite{GraGri11}]\label{thm:ggsh1}
	For any $q \geq 1$ there exists a constant $c > 0$ such that, 
	for any $p\in(0,1)$, any finite graph $G$, any boundary condition $\xi$ and any increasing event $A$,
	\begin{align}\label{eq:ggsh1}
    	\frac{d}{dp}\phi_{p,q,G}^\xi(A)
	  	\geq  c	\,\phi_{p,q,G}^\xi(A) ( 1 - \phi_{p,q,G}^\xi(A))
		\log \left(\frac{1}{ 2m_{A, p}}\right),
  \end{align}
  where 
  $m_{A, p} =\max_{e \in E_G} \big( \phi_{p,q,G}^\xi(A \cond \om(e) = 1) - \phi_{p,q,G}^\xi(A \cond \om(e) = 0) \big).$
\end{theorem}

The original result concerns  a more general class of measures than that of the random-cluster model, 
hence the slightly more complicated statement of \cite[Thm. 5.3]{GraGri11}. 
The above formulation is easily deduced using an explicit bound for the finite energy property of the random-cluster model:
\begin{align*}
 \frac{p}{q} \leq \phi_{p,q,G}^\xi\big(\om(e) =1 \cond \om(f),  f \neq e \big) \leq p, 
 \qquad \text{ for all } G,e,\xi, p \text{ and } q \geq1.&
\end{align*}
  
%
%
In order to state the second result, we introduce the notion of Hamming distance. 
For an event $A$ and a configuration $\om$, define $H_A(\om)$ as the graph distance in the hypercube $\{0,1\}^{E_G}$ 
(or {\em Hamming distance}) 
between $\om$ and the set $A$. 
When $A$ is increasing, it corresponds to the minimal number of edges that need to be turned to open in order to go from $\om$ to $A$.
The following may be found in \cite[Thm.~2.53]{Gri06} or \cite{GriPiz97}.
\begin{theorem}[\cite{GriPiz97}]\label{thm:hamming}
  For any $q\ge 1$, any $p\in(0,1)$, any finite graph $G$ and boundary condition $\xi$, we have that 
  for any increasing event $A$,
  \begin{align}\label{eq:hamming}
    \frac{d}{dp}\log(\phi_{p,q,G}^\xi(A))\ge\frac{\phi_{p,q,G}^\xi(H_A)}{p(1-p)}.  
  \end{align}
\end{theorem}
In the above $\phi_{p,q,G}^\xi(H_A)$ is the expectation of $H_A$ under $\phi_{p,q,G}^\xi$.

\begin{remark}
  In this article~\eqref{eq:hamming} will be used in its integrated form. 
  Consider two values $p'<p$ and an increasing event $A$. 
  Since $H_A$ is a decreasing function, by integrating~\eqref{eq:hamming} between $p'$ and $p$ we find
  \begin{align}\label{eq:hamming_int}
  	\phi_{p',q,G}^\xi(A)\le \phi_{p,q,G}^\xi(A)\exp\big[-4(p - p')\phi_{p,q,G}^\xi(H_A)\big].
  \end{align}	
  Now, consider an event $A$ depending on a finite set of edges $E$ and assume that the infinite-volume measures at $p'$ and $p$ are unique. 
  By taking $\xi=1$ and taking the limit in~\eqref{eq:hamming_int} as $G$ tends to $\Lat$ (both sides of the inequality converge) we obtain 
  \begin{align}\label{eq:hamming_inf}
 	\phi_{p',q}(A)\le \phi_{p,q}(A)\exp\big[-4(p - p')\phi_{p,q}(H_A)\big].
  \end{align}
\end{remark}
\phantom{. }
\medbreak

{\bf From now on, we fix $q\ge1$ and $\calG$. For ease, we drop them from the notation. We will frequently work with infinite-volume measures for different values of $p$ and will always assume that these values are not in $\mathcal D_q$. In such case, $\phi_p$ means the unique infinite-volume measure with parameter $p$.}

\begin{remark} Since $\mathcal D_q$ is countable, the different claims could be easily extended to values of $p$ in $\mathcal D_q$ by density ($\phi_p$ simply denotes any infinite-volume measure in this case). Also note that we are mainly interested in $p<p_c$ for which $p\notin \mathcal D_q$ anyway (we prefer to state the claims in full generality since they may be of some use in other contexts).\end{remark}

\section{Crossings in the easy direction}\label{sec:proof_RSW}

The goal of this section is to prove the following result, which corresponds to Step 1.

\begin{proposition}\label{prop:easy}
  If $p_0\in(0,1)$ is such that there exists an infinite-volume measure $\phi_{p_0}$ with 
  $$\liminf_{n\rightarrow \infty}\phi_{p_0}(\calC_v(2n,n))=0,$$
  then for any $p<p_0$ there exists $c=c(p)>0$ such that for any $x,y\in\calG$
  \begin{align*}
    \phi_{p}(x\longleftrightarrow y)\le \exp(-c|x-y|).
  \end{align*}
\end{proposition}

\begin{remark}\label{rem:step1d>2}
  This proposition can be proved in any dimension $d\ge2$. 
  The claim should be adapted as follows: if the liminf of probabilities of crossing sets of the form $[0,2n]^{d-1}\times[0,n]$ 
  from $[0,2n]^{d-1}\times\{0\}$ to $[0,2n]^{d-1}\times\{n\}$ is equal to 0 for some edge-weight $p_0$, 
  then there is exponential decay for any $p<p_0$ (i.e.~\eqref{eq:exp_decay0} holds for $p<p_0$).  
\end{remark}

The proof of the proposition is based on the following two lemmas. Let $C_x$ be the cluster of the site $x$. 
For simplicity we will henceforth assume $0 \in V_{\Lat}$.

\begin{lemma}\label{lemma:exp_decay}
  Let $p_0>0$. If there exists an infinite-volume measure $\phi_{p_0}$ and $\kappa>0$ such that 
  $\phi_{p_0}(|C_0|^{4+\kappa})<\infty$, then for any $p<p_0$, there exists $c=c(p)>0$ such that for any $n\ge0$,
  \begin{align}\label{eq:exp_decay}
    \phi_{p}(0\longleftrightarrow \partial\Lambda_n)\le \exp(-c n).
  \end{align}
\end{lemma}
It is easy to see that~\eqref{eq:exp_decay} is equivalent to exponential decay, as defined in~\eqref{eq:exp_decay0}.
The previous lemma is classical, see \cite{GriPiz97} or \cite[Thm. 5.64]{Gri06}.
We only mention that its proof is based on the differential inequality~\eqref{eq:hamming}.

\begin{lemma}\label{lem:easy}
  Let $p>p'$. For any $N \ge n$, 
  $$
  \phi_{p'}(\calC_v(2N,N)) 
  \le \exp\left[- (p-p')  \tfrac{N}{n} \Big( 1-\phi_{p}(\calC_v(2n,n))\Big)^{2N/n}\right].
  $$
\end{lemma}

\begin{proof}
  Consider the event $\calC_v(2N,n)$. 
  Any vertical open crossing of $[0,2N] \times [0,n]$ contains at least one of the following:
  \begin{itemize}[nolistsep,noitemsep]
  \item a vertical crossing of a rectangle $[kn,(k+2)n] \times [0,n]$, for some $0\le k<\lfloor N/n \rfloor$,
  \item a horizontal crossing of a square $[kn,(k+1)n] \times [0,n]$, for some $0\le k<\lfloor N/n \rfloor$.
  \end{itemize}
  All the events above have probability bounded from below by $\phi_{p}(\calC_v(2n,n))$. 
  Using the FKG inequality for the complements of these events, we obtain
  \begin{align}\label{eq.covering}
    1-\phi_{p}(\calC_v(2N,n))\ge \big(1-\phi_{p}(\calC_v(2n,n))\big)^{2N/n}.
  \end{align}
  As a consequence, we deduce that 
  $$\phi_{p}(H_{\calC_v(2N,n)})\ge \phi_{p}(H_{\calC_v(2N,n)}\ge1)=1-\phi_{p}(\calC_v(2N,n))\ge \big(1-\phi_{p}(\calC_v(2n,n))\big)^{2N/n}.$$
  Since $\calC_v(2N,N)$ is included in the intersection of $\lfloor \tfrac{N}{n} \rfloor$ translates of $\calC_v(2N,n)$, it follows that
  $$\phi_{p}\left(H_{\calC_v(2N,N)}\right)\ge \floor[\Big]{ \tfrac{N}{n} } \Big(1-\phi_{p}(\calC_v(2n,n))\Big)^{2N/n}.$$
  By~\eqref{eq:hamming_inf} for $p' < p$ we find the result 
  (we have ignored the integer parts in the lemma since $\tfrac1{p(1-p)}\lfloor N/n\rfloor \ge N/n$). 
\end{proof}

The idea of the proof of Proposition~\ref{prop:easy} goes as follows. 
Assuming that $\phi_{p}(\calC_v(2n,n))$ is small for some $n$, we apply Lemma~\ref{lem:easy} repeatedly, 
and obtain a bound on the decay of $\phi_{p-\eps}(\calC_v(2N,N))$ as $N$ increases. A bound on the moments of $|C_0|$ follows. 

\begin{proof}[Proposition~\ref{prop:easy}]  
  Fix some $\ep>0$ and $p_0 > \ep$.
  Let $\alpha > 2$ be a (large) constant, we will see later how to choose it. 
  (We prefer not to give an explicit value for $\alpha$ now, though the requirements for it are universal.)
  Consider a small constant $\delta_0 > 0$, we will see in the proof how to choose $\delta_0$ 
  (its value only depends on $\alpha$ and $\eps$). 
  Assume that there exists a positive integer $n_0$ such that 
  $\phi_{p_0}(\calC_v(2n_0,n_0)) \leq \delta_0^\alpha$ and define recursively, for $k\geq 0$,
  \begin{align*}
    &\delta_{k+1}=\delta_k^2,\\
    &n_{k+1}=n_{k}/\delta_k^2,\\
    &p_{k+1}=p_k-\delta_k.
  \end{align*}
  Assuming $\delta_0$ is sufficiently small, 
  the inequality $\phi_{p_k}(\calC_v(2n_k,n_k))\le \delta_k^\alpha$ and Lemma~\ref{lem:easy} imply
  \begin{align*}
    \phi_{p_{k+1}}(\calC_v(2n_{k+1},n_{k+1}))
    & \le \exp\left(-(p_k-p_{k+1})\frac{n_{k+1}}{n_k}\big(1- \phi_{p_k}(\calC_v(2n_k,n_k)) \big)^{2n_{k+1}/n_k}\right)\\
    & \le \exp\Big(- \frac{(1-\delta_k^\alpha)^{2 \delta_k^{-2}}}{\delta_k}\Big) 
    \le \delta_k^{2\alpha}
    = \delta_{k+1}^{\alpha}.
  \end{align*}
  Further assume that $\delta_0$ is chosen small enough that $\lim_{k\rightarrow \infty}p_k\ge p_0-\ep$. 
  We deduce that for any $k\ge 0$,
  \begin{align}
    \phi_{p_0-\ep}(\calC_v(2n_k,n_k))\le \delta_k^\alpha. 
  \end{align}
  Let us extend the previous bound to values of $N$ different from the $\{n_k : k\ge0\}$. 
  For $n_k\le N < n_{k+1}$, by~\eqref{eq.covering} or by a simple union bound, 
  \begin{align*}
    \phi_{p_0-\ep}(\calC_v(2N,N)) 
    &\le \phi_{p_0-\ep}(\calC_v(2N,n_k)) 
    \le \frac{2 n_{k+1}}{n_k} \delta_k^\alpha 
    \le 2\left(\frac{n_0}{N}\right)^{\frac{\alpha-2}{4}}.
  \end{align*}
  In the last inequality we have used that $\delta_i=\delta_k^{1/2^{k-i}}$ for any $i\le k$, 
  and therefore $$\delta_k^4\le \prod_{i=0}^k\delta_i^2= \frac{n_0}{n_{k+1}}\le \frac{n_0}N.$$ 
  
  It easily follows that
  $\phi_{p_0-\ep}(0\leftrightarrow \partial\Ball_N) \le 8\left(\frac{n_0}{N}\right)^{\frac{\alpha-2}4}$ 
  for any $N$.
  Since a cluster of cardinality larger than $N$ has diameter at least a constant times $\sqrt N$, 
  we find easily that $\phi_{p-\ep}(|C_0|^5)<\infty$ provided that $\alpha$ is chosen large enough
  ($\alpha >42$ suffices).
  By Lemma~\ref{lemma:exp_decay}, the above implies that for any $p<p_0-\ep$, there exists $c=c(p)>0$ such that for any $n\ge0$,
  $$\phi_p(0\longleftrightarrow\partial\Lambda_n)\le \exp(-cn).$$
  Now, if $\liminf\phi_{p_0}(\calC_v(2n,n))=0,$ then for any $\ep>0$, there exists $n_0$ such that 
  $\phi_{p_0}(\calC_v(2n_0,n_0)) \leq \delta_0^\alpha$, where $\alpha$ and $\delta_0=\delta_0(\ep, \alpha)$ are chosen as above. 
  By the argument above $\phi_{p}$ exhibits exponential decay for any $p<p_0$.
\end{proof}

\section{Crossing probabilities in the hard direction}  \label{sec:proof_RSW2}

The object of this section is the following result.
\begin{proposition}\label{prop:RSW}
  If $p\in(0,1)$ is such that there exists an infinite-volume measure $\phi_{p}$ with 
  \begin{align}\label{eq:RSW}
    \liminf_{n\rightarrow \infty}\phi_{p}(\calC_v(2n,n))>0,
  \end{align}
  then for any $p_0>p$,
 \begin{align*}
    \liminf_{n\rightarrow \infty}\phi_{p_0}(\calC_h(2n,n))>0.
  \end{align*}
\end{proposition}
In light of Proposition~\ref{prop:easy}, the above result has the following immediate corollary, 
which is exactly the claim mentioned in Step 2 of the introduction.

\begin{corollary}\label{cor:RSW}
  If $p_0\in(0,1)$ is such that there exists an infinite-volume measure $\phi_{p_0}$ with 
  \begin{align*}
    \liminf_{n\rightarrow \infty}\phi_{p_0}(\calC_h(2n,n))=0,
  \end{align*}
  then for any $p<p_0$ there exists $c=c(p)>0$ such that for any $n\geq 0$,
  \begin{align*}
    \phi_{p}(0\longleftrightarrow \partial\Lambda_n) \leq e^{-cn}.
 \end{align*}
\end{corollary}

The proof of Proposition~\ref{prop:RSW} is based on the following lemma and its corollary.
Some terminology is needed for their statement. 
Let $\ga^1, \dots, \ga^K$ be open paths in some rectangle $[a,b]\times [c,d]$.
We say they are \emph{separated} in $[a,b]\times [c,d]$
if they are contained in distinct clusters of $[a,b]\times [c,d]$ 
(beware of the fact that we are speaking of clusters {\em in} $[a,b]\times[c,d]$).
In other words, no two are connected by open paths inside $[a,b]\times[c,d]$.
\begin{lemma}\label{lem:manycross}
  Let $p \in (0,1)$ and $n \in \bbN$. 
  There exist universal constants $c_0, c_1 >0$ such that, 
  if $1 \leq I \leq n/400$ is an integer that satisfies 
  \begin{align}\label{eq.u_bound}
    I^2 \leq c_0 \frac{ \phi_p\big( \calC_v(2n,n) \big)} { \phi_p(\calC_h(2n,n))^{c_1 / I}},
  \end{align}
  then 
  \begin{align}\label{eq.manycross}
    \phi_p \left([0,2n] \times [0,n/2] \mathrm{\ has\ }2^{I}\mathrm{\ separated\ vertical\ crossings}\right)
    \geq \tfrac{1}{2}\phi_p\big( \calC_v(2n,n) \big).
  \end{align}
\end{lemma}

The statement above may seem cryptic. 
Here are a few observations that may help the reader assimilate the lemma. 
First of all, the conclusion~\eqref{eq.manycross} is strongest when $I$ is large, 
but the hypothesis~\eqref{eq.u_bound} is effectively an upper bound on $I$. 
Moreover it may even be that there exists no $I$ with the properties required in the lemma. 
The lemma will be applied in situations where $\phi_p\big( \calC_v(2n,n) \big)$ is bounded below by some constant. 
Then it states that, if $\phi_p(\calC_h(2n,n))$ is close to $0$ (so that $I$ may be large and satisfy~\eqref{eq.u_bound}),
the rectangle $[0,2n] \times [0,n/2]$ contains many separated vertical crossings with positive probability. 
Furthermore, the smaller $\phi_p(\calC_h(2n,n))$, the larger the number of separated vertical crossings. 

In words, this statement asserts that if typically $[0,2n] \times [0,n]$ is crossed vertically, 
but the probability of crossings in the hard direction is very small, 
then any vertical crossing needs to twist substantially, 
creating many separated crossings of a slightly smaller (in height) rectangle 
(see the discussion preceding the proof of Lemma~\ref{lem:manycross}).

The proof of Lemma~\ref{lem:manycross} represents the major difficulty of this article. 
We postpone it to the end of the section and first explain how it implies Proposition~\ref{prop:RSW}. 
A key observation is that the existence of separated vertical crossings of $[0,2n]\times[0,n/2]$ (as in~\eqref{eq.manycross})
implies a lower bound on the Hamming distance to the event $\calC_h(2n,n)$. 
Using~\eqref{eq:hamming}, this yields an explicit lower bound on crossing probabilities in the hard direction.
We formalize this next. 
\medbreak
For $x\in(0,1)$, set 
$$f(x):=\begin{cases}\frac{\log(1/x)}{\log\log(1/x)}&\text{ if $x<1/e$}\\ -\infty&\text{ otherwise}\end{cases}.$$ 

\begin{corollary}\label{cor:crucial}
  Let $\delta > 0$. 
  There exist constants $c_2=c_2(\delta)>0$ and $c_3=c_3(\delta)>0$ such that for any $p > p' $ and $n$ with
  $\phi_p \big( \calC_v(2n,n) \big) \ge \delta$ and 
  $n \geq c_2 f\big[\phi_p \big( \calC_h(2n,n) \big)\big]$, 
  the following holds:
  \begin{align*}
    \phi_{p'} \big( \calC_h(n,n/2) \big) 
    \le \exp\Big[ -c_3 (p-p') \delta \exp \Big( c_3 f\big[\phi_p \big( \calC_h(2n,n) \big)\big] \Big)\Big].
  \end{align*}
\end{corollary}


\begin{proof}
  Fix $\delta >0$ and $p> p'$. 
  Let $n$ be an integer such that $\phi_p\big( \calC_v(2n,n) \big) \ge \delta$
  and $  n \geq c_2 f\big[\phi_p \big( \calC_h(2n,n) \big)\big]$, for a constant $c_2$ specified later.
  Define $I = \lfloor \cf f[\phi_p(\calC_h(2n,n))] \rfloor$,
  where $\cf=\cf(\delta)$ is some large constant to be specified.
  It is easy then to see that, for this choice of $I$, we have 
  $$ I^2 \leq c_0 \frac{ \phi_p\big( \calC_v(2n,n) \big)} { \phi_p(\calC_h(2n,n))^{c_1 / I}}$$
  for every $n \ge 1$, provided that $\cf$ is large enough 
  (where $c_0,c_1$ are the universal constants of Lemma~\ref{lem:manycross}).
  Furthermore, we find that $I \leq n/400$ by setting $c_2 =400 \cf$.
  Finally, we may limit ourselves to the case where $\phi_p\left(\calC_h(2n,n)\right)$ is small enough to have $I \ge 1$
  (the constant $c_3$ may be chosen so that the conclusion holds trivially otherwise). 

  The previous paragraph shows that with these choices of $\cf$ and $c_2$,
  $I$ satisfies the assumptions of Lemma~\ref{lem:manycross}, and we find 
   \begin{align*}
     \phi_p \left([0,2n] \times [0,n/2] \mathrm{\ contains\ }2^I\mathrm{\ separated\ vertical\ crossings}\right) 
     \geq \frac{1}{2}\phi_p\big( \calC_v(2n,n) \big)\ge \frac{\delta}2.
  \end{align*}
  Since the crossings in~\eqref{eq.manycross} are separated,
  there exist also at least $2^I - 1\ge 2^{I-1}$ disjoint dual vertical crossings of $[0, 2n]\times [0,n/2]$.
  This generates a lower bound on the expected Hamming distance to the event $\calC_h(2n, n/2)$:
  \begin{align}\label{eq:ham_low}
    \phi_p \left(H_{\calC_h(2n,n/2)}\right) \geq 2^{I}\cdot\frac{\delta}4 .
  \end{align}
  Inequality~\eqref{eq:hamming_inf} (the integrated form of~\eqref{eq:hamming}) implies that
  \begin{align*}
    \phi_{p'}( \calC_h(2n,n/2) )
    & \le \phi_{p} ( \calC_h(2n,n/2) ) \exp \Big( - \frac{1}{p(1-p)}\cdot(p-p')\cdot 2^{I}\cdot\tfrac{\delta}4 \Big)\\
    &\le \exp \Big[ - (p-p')\delta \exp\Big(c_3 f\big[\phi_p\big(\calC_h(2n,n)\big)\big]\Big) \Big], 
  \end{align*}
  where the constant $c_3>0$ depends on $\cf$ and therefore on $\delta$ only. In the last inequality, we used the choice of $I$ proposed at the very beginning of the proof.
  Finally, by combining crossings in the hard direction of five rectangles with side lengths $n$ and $n/2$,
  we may obtain a crossing of $[0,2n]\times [0,n/2]$. Thus, 
  $$\phi_{p'}\left(\calC_h(n,n/2)\right)^5 \le \phi_{p'}\left(\calC_h(2n,n/2)\right),$$
  and the result follows. 
\end{proof}
Let us now prove Proposition~\ref{prop:RSW} using Corollary~\ref{cor:crucial}.
\begin{proof}[Proposition~\ref{prop:RSW}]
  Fix $p_0>p$ and assume that  $\displaystyle\inf_{n\ge0}\phi_{p}(\calC_v(2n,n))=\delta>0$.
  Let $c_2, c_3$ be the constants given by Corollary~\ref{cor:crucial} for $\de$ given above. 
  For integers $n_0$ and $0 \leq k \leq  \log_2 \sqrt{n_0} $, define 
  \begin{align*}
    n_k &= 2^{-k} n_0,\\ 
    p_k &= p_0 - (p_0-p) \sum_{i=1}^k 2^{-i},\\   
    \beta_k &= \phi_{p_k}(\calC_h(2n_k,n_k)).
  \end{align*} 
  We aim to apply Corollary~\ref{cor:crucial} with the above values of $n$ and $p$. 
  We start by a simple verification of the hypothesis. 
  
 \begin{claim*}
  {\em  For $n_0$ large enough and for any integer $0\le k\le \log_2 \sqrt{n_0} $,}
  $$n_k > c_2 f(\beta_k).$$
\end{claim*}

  \noindent{\em Proof of the Claim.}
  Assume that there exists an integer $0\le k\le \log_2 \sqrt{n_0}$ such that $n_k \leq c_2 f(\beta_k)$.
  We have 
  \begin{align}\label{eq.longcross}
    \phi_{p}(\calC_h(2n_k,n_k)) 
    \le \phi_{p_k}(\calC_h(2n_k,n_k)) = \beta_k
    \leq \exp \Big(- \frac{n_k}{c_2} \Big)
    \leq n_k^{-10},
  \end{align}
  where the second inequality uses that $f(x) \le \log(1/x)$.
  In the last inequality we have supposed that $n_k$ is larger than some rank depending only on $c_2$. 
  We may assume this since $n_k \geq \sqrt n_0$ and we may take $n_0$ as large as we wish. 

  Consider $ x \in \{0\} \times [0, n_k]$ and $y\in \{\tfrac12 n_k\} \times [0, n_k]$
  maximizing (among such pairs of vertices) the probability that they are connected in $[0,\tfrac12 n_k]\times[0,n_k]$.
  Then 
  $$ 
  \phi_{p} \left( x \xlra{[0,n_k/2]\times[0, n_k]} y \right)
  \geq \frac 1{n_k^2}\phi_{p}\left( \calC_h(\tfrac12 n_k, n_k) \right).
  $$
  Combining four times the above (also using reflection symmetry) we obtain
  $$\phi_{p}(\calC_h(2n_k,n_k)) \ge\frac1{n_k^8}\phi_{p}(\calC_h(\tfrac12 n_k,n_k))^4.$$
  Confronting this to~\eqref{eq.longcross} implies
  \begin{align*}
    \phi_{p}\left( \calC_h(\tfrac12 n_k, n_k) \right) \leq n_k^{-1/2} \leq n_0^{-1/4}.
  \end{align*}
  But $\phi_{p}\left(\calC_h(\tfrac12 n_k, n_k) \right) \ge \delta $ by assumption and symmetry under $\tfrac\pi2$-rotation.
  This leads to a contradiction for $n_0$ large enough. 
  \medbreak
  The argument that~\eqref{eq.longcross} contradicts  $\phi_{p}\left(\calC_h(\tfrac12 n_k, n_k) \right) \ge \delta $ 
  will be used several times in the rest of the paper. \begin{flushright}$\diamond$\end{flushright}
  
  We now fix $n_0$ large, in particular large enough for the property of the claim to be satisfied. 
  Then we may apply Corollary~\ref{cor:crucial} to each triplet $(n_k,p_k,p_{k+1})$ to obtain
  \begin{align*}
    \beta_{k+1}  
     \le \exp\Big(-c_3 2^{-(k+1)} (p_0-p) \de \exp\big(c_3f(\beta_k)\big)\Big) .
  \end{align*}
  Hence, there exist constants $ \Delta \geq e^{40}$ and $c_{\Delta} >0$, depending on $p_0-p$, $c_3$ and $\de$ only,
  such that if we assume $\beta_k \leq c_{\Delta} \Delta^{-k}$, the previous displayed equation implies  that
  \begin{align*}
    & c_3 f(\beta_k) \geq 2 k \log 2
    \quad \text{ and} \quad
    \beta_{k+1} 
    \leq  \exp\bigg[-c_3 (p_0-p) \de \exp\Big( \frac{c_3}2 f(\beta_k) \Big)\bigg] \leq \frac{\beta_k}{\Delta} 
    \leq c_{\Delta} \Delta^{-(k+1)}.
  \end{align*}
  Assume now that $\beta_0 \leq c_{\Delta}$.
  Then, by the above, $\beta_k \leq c_{\Delta} \Delta^{-k}$ for any $k \leq \log_2 \sqrt{n_0} $.
  Therefore, there exists $m\in[\sqrt n_0,n_0]$ ($m = n_{\lfloor \log_2 \sqrt{n_0} \rfloor}$) such that
  $$  \phi_{p}\left( \calC_h(2m,m) \right) \leq c_{\Delta} e^{-40 \lfloor \log_2 \sqrt{n_0} \rfloor} \leq  c_\Delta m^{-10}.$$
  Using the same procedure as at the end of the proof of the previous claim
  we obtain a contradiction for $n_0$ large enough, 
  since $m \geq \sqrt{n_0}$ and $\phi_{p}\left( \calC_h(2m,m) \right)\ge\delta$ by definition. 
 
  Therefore, the assumption $\phi_{p_0}(\calC_h(2n_0,n_0)) = \beta_0 \leq c_{\Delta}$ can not hold for $n_0$ large enough. 
  This implies that 
  $$\liminf_{n\to \infty}\phi_{p_0}(\calC_h(2n,n)) \geq c_{\Delta} >0.$$
\end{proof}

We now turn to the core of the argument, namely the proof of Lemma~\ref{lem:manycross}. 
The proof is inspired by the work of Bollob\'as and Riordan on Bernoulli percolation on Voronoi tessellations \cite{BolRio06}
(even though it makes use of different ingredients, and that the claim is not the same). 
We start with a brief description. 

Fix $I$ as in Lemma~\ref{lem:manycross} and let $v = \frac{1}{100 I}$. 
First we obtain an upper bound, as a function of $\phi_p(\calC_h(2n,n))$,
for the probability of crossing horizontally rectangles of height $k$ and width ${(1+v)k}$ for $k \in [\frac{n}4,n]$.
Using this bound, we show that one vertical crossing of $[0,2n] \times [0,n]$ contains, with high probability, 
three crossings of the slightly thinner rectangle $[0,2n] \times [23v n,(1 - 23v )n]$.
Repeating the procedure, we finally obtain $2^I$ crossings of $[0,2n] \times [\frac{n}{4},\frac{3n}{4}]$.
Moreover, these crossings are separated by dual paths. See Figure~\ref{fig:onemakesmany}.

\begin{proof}[Lemma~\ref{lem:manycross}]
  Fix $p,n$ and $I$ satisfying the assumptions of the lemma and set $v = \frac{1}{100 I}$ 
  (we will specify the values of the universal constants $c_0$ and $c_1$ later in the proof). 
  Define
  \begin{align}\label{eq:max}
    \alpha = \sup \Big\{ \phi_p\big( \calC_h(\lceil(2+v)k \rceil, 2k) \big) : k\in[\tfrac n8,\tfrac n2]\Big\}.
  \end{align}
  For any $k\in[\frac n8,\frac n2]$, 
  we may combine $32/v$ crossings in the hard direction 
  of rectangles with sides of length $2k$ and $\lceil(2+v)k \rceil$ (both horizontal and vertical)
  to create a horizontal crossing of $[0,2n] \times [0,n]$.
  Choosing $k\in[\frac n8,\frac n2]$ achieving the maximum in~\eqref{eq:max}, we conclude that 
  \begin{equation}\label{eq:bound alpha}
    \alpha \le \phi_p(\calC_h(2n,n))^{v/32} \le \phi_p(\calC_h(2n,n))^{2c_1/I},
  \end{equation} 
  by setting $c_1 = 1/6400$. 
  
  We start by proving a series of claims that will then be used to prove the lemma.
  For these claims, fix an integer $k \in [\frac n4,\frac n2]$ and $u \in [v, 1/3] $ such that $ku \in \bbZ$. 
  The first three claims are concerned with crossings of the rectangle $\bfR(k) = [-(1 + u)k, (1 +u) k ] \times [0,2k]$.
  
  \begin{claim}\label{startfromcenter} 
    Let $\calE(k)$ 
    be the event that there exists a vertical open crossing of $\bfR(k)$,
    with the lower endpoint not contained in $[-3uk, 3uk] \times \{0\}$, 
    or the higher endpoint not contained in $[-3uk, 3uk] \times \{2k\}$.
    Then $$ \phi_p (\calE(k)) \leq 4(\alpha + \sqrt{\alpha}). $$
  \end{claim}

  \noindent{\em Proof of Claim 1.}
    Let $\beta$ be the $\phi_p$-probability that there exists a vertical open crossing of $\bfR(k)$,
    with the lower endpoint in $[-(1 + u)k,-3uk] \times \{0\}$.

    \begin{figure}
      \begin{center}
        \includegraphics[width=0.5\textwidth]{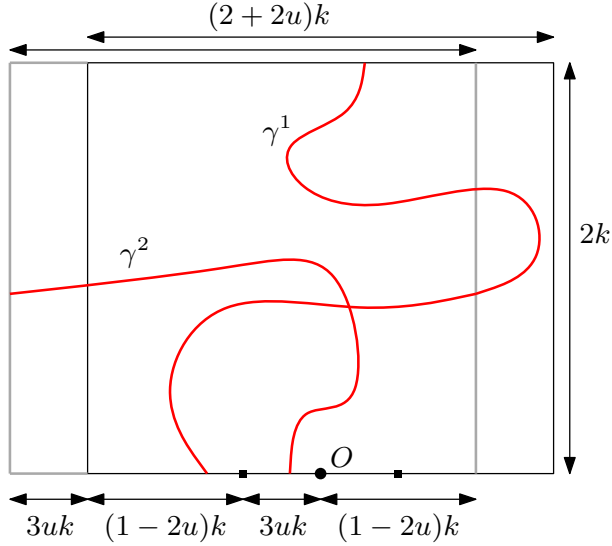}
      \end{center}
      \caption{In the black rectangle $\bfR(k)$, a path $\ga^1$ connects $[-(1 + u)k,-3uk] \times \{0\}$ to the top side.
        Except on an event of probability $\alpha$, $\ga^1$ crosses the vertical line $\{(1 - 2u)k\} \times [0,2k]$ (grey). 
        By reflection we may construct a path $\ga^2$, contained in $[-(1 + 4u)k, (1 - 5u)k] \times [0,2k]$, 
        and connecting $[ - 3uk ,(1 - 5u)k]\times \{0\}$ and $\{-(1 + 4u)k\} \times [0,2k]$. 
        The two induce a horizontal crossing of the grey rectangle $[-(1 + 4u)k, (1 - 2u)k] \times [0,2k]$. }
      \label{fig:startfromcenter}
    \end{figure}
    
    The probability of crossing $[-(1 + u)k, (1 - 2u)k] \times [0,2k]$ vertically is at most $\alpha$ (by definition of $\alpha$).
    Thus, with probability $\beta - \alpha$, there exists a vertical crossing of $\bfR(k)$ 
    with an endpoint in $[-(1 + u)k,-3uk] \times \{0\}$ which intersects the vertical line $\{(1 - 2u)k\} \times [0,2k]$.
    See Figure~\ref{fig:startfromcenter}.
    By reflection with respect to $\{-3uk\}\times [0,2k]$,
    with probability $\beta -\alpha$, there exists an open path in $[-(1 + 4u)k, (1 - 5u)k] \times [0,2k]$, 
    between $[ - 3uk ,(1 - 5u)k]\times \{0\}$ and $\{-(1 + 4u)k\} \times [0,2k]$.
    
    When combining the two events above using the FKG inequality, we obtain that, 
    with probability at least $(\beta -\alpha)^2$, 
    there exists a horizontal open crossing of $[-(1 + 4u)k, (1 - 2u)k] \times [0,2k]$.
    This event has probability less than $\alpha$, hence
    $ \beta \leq \alpha + \sqrt{\alpha}.$
    By considering the other possibilities for the lower and higher endpoints, the claim follows.
    \begin{flushright}$\diamond$\end{flushright}
    
    \begin{claim}\label{touchright} 
      Let $\calF(k)$ 
      be the event that there exists a vertical open crossing of $\bfR(k)$
      that does not intersect the vertical line $\{(1 - 2u)k \} \times [0,2k]$.
      Then  $$ \phi_p (\calF(k)) \leq 2\alpha. $$
    \end{claim}

  \noindent{\em Proof of Claim 2.}
    Any vertical crossing of $\bfR(k)$ not touching $\{(1-2u)k \} \times [0,2k]$ 
    is either contained in  $[-(1 + u)k, (1 - 2u) k ] \times [0,2k]$ or $[(1 - 2u)k, (1 + u) k ] \times [0,2k]$.
    Both these rectangles are crossed vertically with probability less than $\alpha$, and the claim follows (we used that $u\le 1/3$). 
    \begin{flushright}$\diamond$\end{flushright}

    \begin{claim}\label{topbeforeright} 
      Let $\calG(k)$ 
      be the event that there exists an open path in $\bbR \times [0,(2- 11 u)k]$
      between $[-3uk, 3uk] \times \{0\}$ and the vertical segment $\{(1-2u)k\} \times  [0,(2- 11 u)k]$.
      Then  $$ \phi_p (\calG(k)) \leq \alpha + \sqrt\alpha.$$
    \end{claim}
    
  \noindent{\em Proof of Claim 3.}
    Let $\beta = \phi_p (\calG(k))$. 
    Suppose $\calG(k)$ occurs and let $\ga$ be an open path in $\bbR \times [0,(2- 11 u)k]$
    between $[-3uk, 3uk] \times \{0\}$ and $\{(1-2u)k\} \times  [0,(2- 11 u)k]$.
    There are two possibilities for $\ga$. 
    Either $\ga$ crosses the line $\{-(1-8u)k\} \times [0,(2-11u)k]$, or it does not. 

    The first situation arises with probability at most $\alpha$, 
    since it induces a horizontal crossing of the rectangle $[-(1-8u)k, (1-2u)k] \times [0, (2 - 11u)k]$.
    See the left diagram in Figure~\ref{fig:topbeforeright}.

    \begin{figure}
      \begin{center}
        \includegraphics[width=1.0\textwidth]{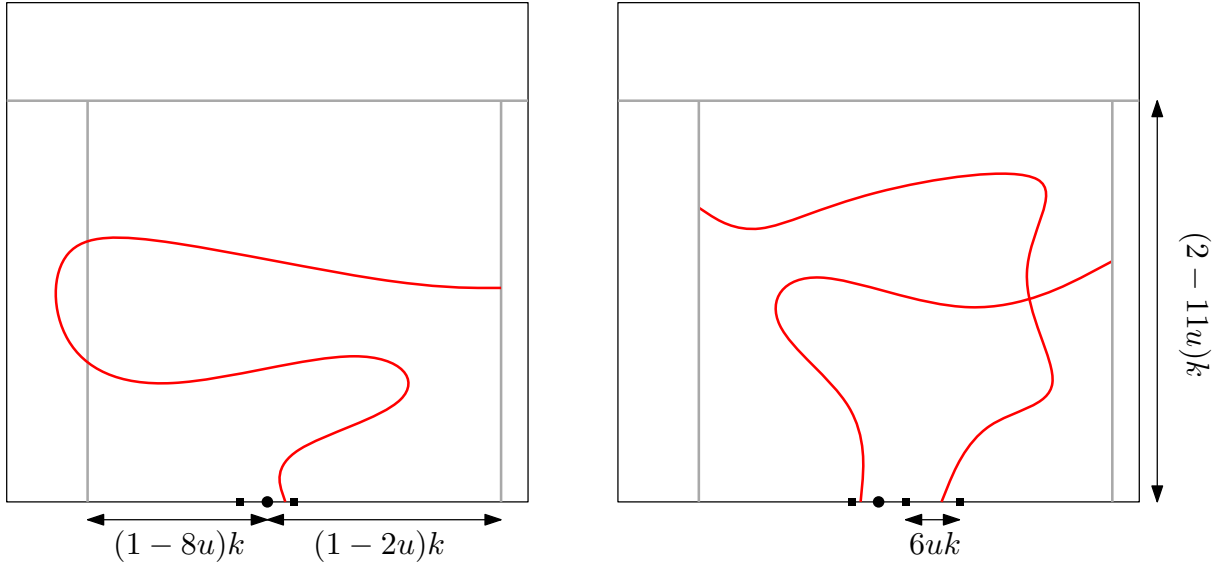}
      \end{center}
      \caption{The two possibilities for the path $\ga$. 
        The black rectangle $\bfR(k)$ is depicted for scaling purposes, 
        and the strip  $\bbR \times [0,(2- 11 u)k]$ is delimited by the top grey line. 
        The origin is marked by a disk. 
        \emph{Left:} the first situation, the path $\ga$ crosses the vertical line $\{-(1-8u)k\} \times [0,(2- 11 u)k]$.
        \emph{Right:} two occurrences of the second situation may be used to create a horizontal crossing of
        $[-(1-8u)k, (1-2u)k] \times [0, (2 - 11u)k]$.}
      \label{fig:topbeforeright}
    \end{figure}
    
    Thus the second situation arises with probability at least $\beta - \alpha$. 
    Then, by symmetry with respect to $\{3uk\} \times \bbR$ and the FKG inequality, 
    with probability at least $(\beta - \alpha)^2$, $[-(1-8u)k, (1-2u)k] \times [0, (2 - 11)k]$ contains two open paths:
    \begin{itemize}[noitemsep,nolistsep]
    \item one connecting $[-3uk, 3uk] \times \{0\}$ to $\{(1-2u)k\} \times  [0,(2- 11u)k]$,
    \item one connecting $[3uk, 9uk] \times \{0\}$ to $\{-(1-8u)k\} \times  [0,(2- 11 u)k]$.
    \end{itemize}
    These two paths induce an open horizontal crossing of $[-(1-8u)k, (1-2u)k] \times [0, (2-11u)k]$, 
    thus $(\beta - \alpha)^2 \leq \alpha$, and the claim follows. 
    \begin{flushright}$\diamond$\end{flushright}
  
  In the claims above we have defined the events $\calE(k)$, $\calF(k)$ and  $\calG(k)$. 
  In addition, define $\widetilde{\calG}(k)$ as the symmetric of $\calG(k)$ with respect to the line $\bbR \times \{ k \}$, 
  i.e. the event that there exists an open path in $\bbR \times [11uk,2k]$
  between $[-3uk, 3uk] \times \{2k\}$ and $\{(1-2u)k\} \times  [11uk,2k]$.
  The bound of Claim~\ref{topbeforeright} applies to $\widetilde{\calG}(k)$ as well. 

  All four events revolve around the rectangle $\bfR(k)$.
  In the following, we will use translates of these events, 
  and we will say for instance that $\calE(k)$ occurs in some rectangle $\bfR(k) + z$ 
  if $\calE(k)$ occurs for the translate of the configuration by $-z$. 
  
  \begin{claim}\label{onemakesthree}
    Except on an event $\calH(k)$ of probability at most $\frac1 u (54\alpha + 36\sqrt{\alpha})$, 
    any open vertical crossing of $\bfS(k) = [0,2n] \times [-k, k]$, 
    contains two separated vertical crossings of $ \bfS((1 - 11u)k) =  [0,2n] \times [-(1 - 11u)k, (1 - 11u)k]$.
  \end{claim}
  
  \noindent{\em Proof of Claim 4.}
  The rectangle $[0,2n] \times [-k, k]$ is the union of the rectangles
  $\bfR_j = [juk,(2 + (j + 2)u)k] \times [-k, k]$, for $0\le j\le J$, 
  where 
  $$J ~:=~  \big\lfloor\tfrac1u (\tfrac{n}{k} -2)\big\rfloor - 2 ~\leq~ 6/u.$$
  Let $\calH(k)$ be the union of the following events for $0\le j\le J$: 
  \begin{itemize}[nolistsep,noitemsep]
  \item the rectangle $[juk,(2 + (j + 1)u)k] \times [-k, k]$ contains a horizontal open crossing,
  \item $\calE(k)$ occurs in the rectangle $\bfR_j$,
  \item $\calF(k)$ occurs in the rectangle $\bfR_j$,
  \item at least one of $\calG(k)$ and $\tilde{\calG}(k)$ occurs in the rectangle $\bfR_j$.
  \end{itemize}
  Using a simple union bound and the estimates of Claims~\ref{startfromcenter}-\ref{topbeforeright}, we obtain 
  \begin{align}\label{eq:1001}
    \phi_p( \calH(k) ) 
    \leq \frac{100 \sqrt{\alpha}}{u}.  
  \end{align}
  Consider a configuration not in $\calH(k)$ containing a vertical open crossing $\ga$ of $\bfS(k)$. 
  We are now going to explain why such a crossing necessarily contains two 
  (in fact even three but we will not use this fact here) separated crossings of $\bfS((1-11u)k)$. 
  We recommend that the reader takes a look at Figure~\ref{fig:onemakesthree} first.
    
  Since none of the rectangles $[juk,(2 + (j + 1)u)k] \times [-k, k]$ is crossed horizontally,
  $\ga$ is contained in one of the rectangles $\bfR_j$.
  Fix the corresponding index $j$. Parametrize $\ga$ by $[0,1]$, with $\ga_0$ being the lower endpoint. 
  
  \begin{figure}
    \begin{center}
      \includegraphics[width=0.45\textwidth]{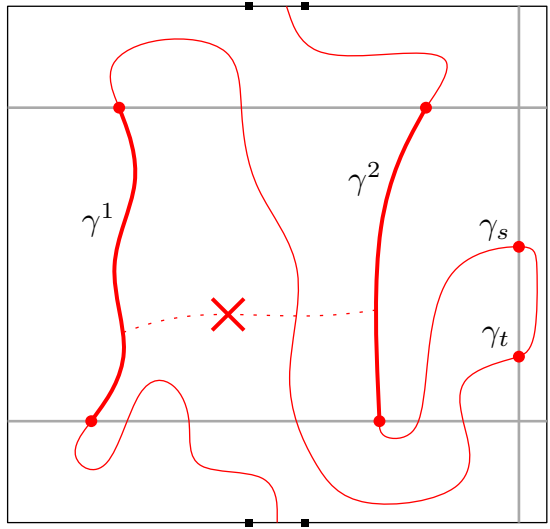}
    \end{center}
    \caption{The path $\ga$ under the assumption that
      $\calE(k)$, $\calF(k)$, $\calG(k)$ and $\tilde{\calG}(k)$ do not occur in $\bfR_j$.
      The endpoints are contained in segments of length $6uk$ around the centres of the top and bottom sides. 
      A first crossing of the strip $\bfS((1 -11 u)k)$ occurs between $\ga_0$ and $\ga_t$, 
      and a second between $\ga_s$ and $\ga_1$. 
      The two crossings $\ga^1, \ga^2$ need to be separated 
      to ensure that $\calF(k)$, $\calG(k)$ and $\tilde{\calG}(k)$ do not occur.}
    \label{fig:onemakesthree}
  \end{figure}
  
  Since $ \calE(k)$ does not occur in $\bfR_j$, $\ga_0$ and $\ga_1$, are contained in
  $[(1+(j-2)u)k,(1+(j+4)u)k] \times \{-k\}$ and  $[(1+(j-2)u)k,(1+(j+4)u)k] \times \{k\}$, respectively. 
  Moreover, since $\calF(k)$ does not occur in $\bfR_j$, 
  $\ga$ crosses the vertical line $\{(2 + (j-1)u)k \} \times [-k, k]$. 
  Let $t$ and $s$ be the first and last times that $\ga$ intersects this vertical line. 
  
  Since $\calG(k)$ does not occur in $\bfR_j$, 
  $\ga$ intersects the line $[0,2n] \times \{(1-11u)k \}$ before time $t$. 
  Likewise, since $\tilde{\calG}(k)$ does not occur, 
  $\ga$ intersects the line $[0,2n] \times \{-(1-11u)k \}$ after time $s$. 
  This implies that $\ga$ contains at least two disjoint crossings of $\bfS((1 - 11u)k)$.
  Call $\ga^1$ the first one (in the order given by $\ga$) and $\ga^2$ the last one. 
  
  The above holds for any vertical crossing $\ga$ of  $\bfS(k)$, 
  hence the crossings $\ga^1$ and $\ga^2$ are necessarily separated in $\bfS((1 - 11u)k)$.
  Indeed, if they were connected inside  $\bfS((1 - 11u)k)$, then $\calF(k)$ would occur.
  \begin{flushright}$\diamond$\end{flushright}
    
  \begin{remark}
    It is actually possible to prove that, in the situation described above, 
    $\ga$ contains at least three separated vertical crossings of $\bfS((1 - 11u)k)$.
    We do not detail this as it is not essential for our proof, 
    but the situation will be depicted in the relevant figures. 
  \end{remark}
  
  Getting back to the proof of the lemma.
  Let 
  $k_i = \lfloor (1 - 22vi)n/2 \rfloor $ for $0\le i\le I$.
  We will investigate vertical crossings of the nested strips $\bfS(k_i) =[0,2n] \times [-k_i, k_i]$.
  Note that $\bfS(k_0)$ is contained in a translation of the rectangle $[0,2n] \times [0,n]$, 
  and that $\bfS(k_I)$ contains a translation of the rectangle $[0,2n] \times [0,n/2]$.
  
  Fix a sequence $(u_i)_i$, 
  with $u_i \in [v, 2v]$ and $k_i u_i \in \bbZ$ for $ 0\le i < I$. 
  The existence of $u_i$ is due to the fact that $v \ge \tfrac4n$ (since $I \le n/400$).
  Define the events $\calH(k_i)$ of Claim~\ref{onemakesthree} for these values of~$u_i$. 
  Except on the event $\bigcup_{i = 0}^{I-1}  \calH(k_i)$, 
  any vertical crossing of $\bfS(k_0)$ generates $2^I$ 
  vertical open crossings of $\bfS(k_I)$ which are separated in $\bfS(k_I)$.
  Indeed, by Claim~\ref{onemakesthree}, every crossing of $\bfS(k_i)$ contains two separate crossings of 
  $\bfS(k_{i+1})\subset \bfS((1 - 11 u_i)k_{i}) $.
  See Figure~\ref{fig:onemakesmany}.  
  By the union bound and Claim~\ref{onemakesthree},
  \begin{align*}
    \phi_p \left( \bigcup_{i = 0}^{I - 1}  \calH(k_i) \right)  
    \leq \frac{100\sqrt{\alpha}}{u}I\le 10\,000 \sqrt{\alpha} I^2
    \leq \frac{\phi_p\big( \calC_v(2n,n) \big)}2,
  \end{align*}
  where the second inequality is due to the choice of $I$ 
  and the fact that we may assume $c_0 \leq \frac{1}{20\,000}$
  (see~\eqref{eq.u_bound} and~\eqref{eq:bound alpha}).
  But $\bfS(k_0)$ is crossed vertically with probability at least $\phi_p\big( \calC_v(2n,n) \big)$.
  Thus, with probability at least $\phi_p\big( \calC_v(2n,n) \big)/2$, $\bfS(k_I)$ contains $2^I$ separated vertical crossings.
  The claim follows from the fact that $\bfS(k_I)$ contains a translate of $[2n,n/2]$.
\end{proof}

\begin{figure}
  \begin{center}
    \includegraphics[width=0.7\textwidth]{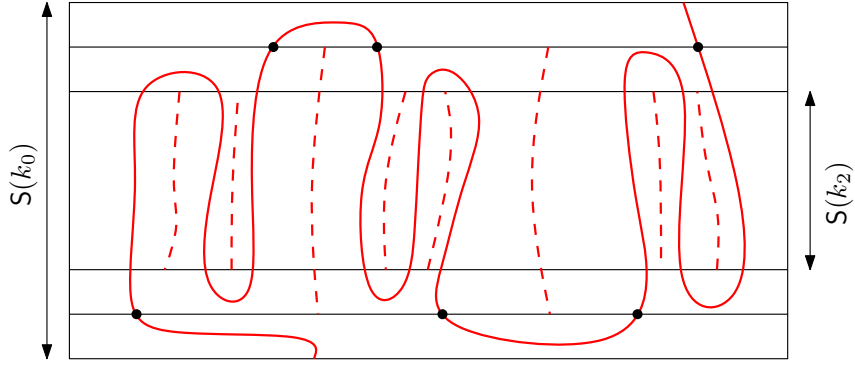}
  \end{center}
  \caption{Under $\left( \cup_{i = 0}^{I - 1}  \calH(k_i)\right)^c$, one vertical crossing of $\bfS(k_0)$ 
    contains two (in fact even three) separated crossings of $\bfS(k_1)$ (with marked endpoints).
    Each such crossing contains in turn two (in fact even three) separated crossings of $\bfS(k_2)$.
    This generates four (in fact even nine) separated crossings in $\bfS(k_2)$, 
    and thus three (in fact even eight) dual crossings between them.}
  \label{fig:onemakesmany}
\end{figure}

\section{Proof of Theorem~\ref{thm:main}}\label{sec:mainproof}

The following coupling argument may be used in conjunction with Theorem~\ref{thm:ggsh1} of Graham and Grimmett
to obtain sharp threshold results, as in \cite[Lem. 6.3]{GraGri11}.
In our case the desired result is stated subsequently as a corollary.
For an edge $e$ and a configuration $\om$, write $C_e(\om)$ 
for the open cluster of $e$ in $\om$, i.e. for the union of the open clusters of the endpoints of $e$. 
	
\begin{proposition}\label{prop:ggcoupling}
	Let $G$ be a finite graph, $e \in E_G$ be an edge and $\xi$ be a boundary condition. 
	For $q \geq 1$ and $p \in (0,1)$
	there exists a measure $\Phi$ on $\Om \times \Om$ 
	such that, if $(\pi, \om)$ is distributed according to $\Phi$,
	\begin{itemize}\itemsep-1pt
	\item $\pi$ is distributed according to $\phi_{p,q,G}^\xi( . \cond \pi(e) = 0)$,
	\item $\om$ is distributed according to $\phi_{p,q,G}^\xi( . \cond \om(e) = 1)$,
	\item $\Phi$-almost surely $\pi \leq \om$ and $\pi(f) = \om(f)$ for edges $ f \notin C_e(\om)$. 
	\end{itemize}
\end{proposition}

\begin{corollary}\label{cor:gg_applied}
	For any $0 < p_0 < p_1 < 1$, there exists $c = c(p_0)>0$ such that, for $n \geq 1$, 
	\begin{align}\label{eq:ggsh2}
		\phi_{p_0}(\calC_h(2n,n))
		\big( 1-\phi_{p_1}(\calC_h(2n,n)) \big)
		\leq \big(\phi_{p_1}(0 \lra \pd \Ball_n) \big)^{c(p_1 - p_0)}.
	\end{align}
\end{corollary}

The proposition may be proved by an exploration argument as sketched in \cite{GraGri11} (see also references therein). 
For completeness we provide a proof, then we prove the corollary. 

\begin{proof}[Proposition~\ref{prop:ggcoupling}]
	Fix $G, e, \xi, p$ and $q$ as in the proposition. 
	We follow the coupling between measures presented in the proof of \cite[Proposition 3.28]{Gri06}.
	
	For $f \in E_G$ and $\om \in \Om$
	let $\om^f$ and $\om_f$ be the configurations equal to $\om$ on edges different from $f$, 
	and equal to $1$ and $0$, respectively, on $f$. 
	Also define $D_f(\om)$ to be the indicator function of the event that the endpoints of $f$ 
	are not connected in $\om^\xi \setminus \{f\}$.
	
	Define a continuous time Markov chain on 
	$$S := \Big\{ (\pi, \om) \in  \Om \times \Om : 
	\pi(e) = 0,\, \om(e) = 1, \,
	\pi \leq \om \text{ and }
	\pi(f) = \om(f) \text{ for all $f \notin C_e(\om)$} \Big\} $$
	with generator $J$ given by 
	\begin{align*}
		J(\pi_f , \om ; \pi^f, \om^f) &= 1, \\
		J(\pi , \om^f ; \pi_f, \om_f) &= \frac{1-p}{p} q^{D_f(\om)}, \\
		J(\pi^f , \om^f ; \pi_f, \om^f) &= \frac{1-p}{p}(q^{D_f(\pi)} -  q^{D_f(\om)}),
	\end{align*}
	for all $f \in E_G \setminus \{e\}$. 
	All other non-diagonal elements of $J$ are $0$
	and the diagonal ones are such that 
	$$ \sum_{(\pi',\om')\in S} J(\pi, \om ; \pi', \om') = 0.$$
	It is easy to check that the formula above ensures that, for any $(\om, \pi) \in S$, 
	$J(\om, \pi ; \pi', \om') \neq 0$ only if $(\om', \pi') \in S$. 
	Hence the Markov chain is indeed defined on $S$. 
	It is proved in \cite{Gri06} that this Markov chain has a unique invariant measure 
	which is the desired coupling $\Phi$. 
\end{proof}

\begin{proof}[Corollary~\ref{cor:gg_applied}]
	Fix $0< p_0 < p_1 < 1$ and suppose that there exists a unique infinite-volume measure for each edge-weight $p_0$, $p_1$. 
	We prove the statement for such values of $p_0, p_1$; it extends to all other values by monotonicity. 

	Let $n\geq 1$ and $p \in [p_0, p_1]$.
	Fix a finite subgraph $G$ of $\Lat$ containing $[0,2n] \times [0,n]$ and let $e = (u,v)$ be an edge of $G$. 
	Consider the coupling $\Phi$ of $\phi_{p,q,G}^0(. \cond \om(e) = 0)$ and $\phi_{p,q,G}^0(. \cond \om(e) = 1 )$ 
	given by Proposition~\ref{prop:ggcoupling}.
	Then 
	$$ \phi_{p,q,G}^0\big(\calC_h(2n,n) \cond \om(e) = 1 \big) - \phi_{p,q,G}^0\big( \calC_h(2n,n) \cond \om(e) = 0\big)
	= \Phi \big( \om \in \calC_h(2n,n); \, \pi \notin \calC_h(2n,n)\big).
	$$
	For the event in the right-hand side of the above to occur,
	$C_e(\om)$ must contain a horizontal crossing of $[0,2n] \times [0,n]$. 
	For any choice of $e$, this implies that $C_e(\om)$ has a radius of at least $n$ around $u$. 
	In particular
	\begin{align*}
	 \phi_{p,q,G}^0 \big(\calC_h(2n,n) \cond \om(e) = 1\big) - \phi_{p,q,G}^0\big(\calC_h(2n,n) \cond \om(e) = 0\big)
	& \leq \Phi\big(u\xlra{\om, G} \Ball_n + u\big)+\Phi\big(v\xlra{\om, G} \Ball_n + u\big) \\
	& \leq  c' \phi_{p,q,G}^0 \big(u\lra \pd \Ball_n + u\big).
	\end{align*}
	For the second inequality we have used the finite-energy property of $\phi_{p,q,G}^0$.
	The inequality of Theorem~\ref{thm:ggsh1} may then be written for $p\in[p_0,p_1]$ as
	\begin{align*}
	\frac{d}{dp} \log \bigg[ \frac{\phi_{p,q,G}^0(\calC_h(2n,n))}{1-\phi_{p,q,G}^0(\calC_h(2n,n))} \bigg]
	&\geq c   \log \bigg[ \frac{1}{ \max_{u \in V_G} \phi_{p,q,G}^0(u \lra \pd \Ball_n + u)} \bigg],
	\end{align*}
	where $c >0$ depends on $p_0$ only. 
	Integrating the above between $p_0$ and $p_1$ and keeping in mind that the right-hand side is decreasing in $p$, 
	we obtain, after a short computation,
	\begin{align*}
	\phi_{p_0,q,G}^0(\calC_h(2n,n))\big(1-\phi_{p_1,q,G}^0(\calC_h(2n,n))\big)
	\leq \left( \max_{u\in V_G}\phi_{p_1,q,G}^0(u \lra \pd \Ball_n + u ) \right)^{c (p_1 - p_0)}
	\end{align*}	
	Now as $G$ tends to $\Lat$, both sides of the above converge and we obtain the desired result. 
\end{proof}

The following proposition is standard and will be proved at the end of this section.

\begin{proposition}\label{prop:main}
  Let $p \in (0,1)$ and $\phi_p$ be a random-cluster measure with edge-weight $p$.
  Suppose there exists $c=c(p)>0$ such that for any $u,v\in \calG^*$,
  \begin{align}\label{eq:dual_exp_decay}
    \phi_{p}(u\stackrel{*}{\longleftrightarrow}v)\le \exp(-c|u-v|).
  \end{align}
  Then $\phi_{p}(0 {\lra} \infty) > 0$ and $p \geq p_c$. 
\end{proposition}

%

\begin{proof}[Theorem~\ref{thm:main}]
	Recall the definition of the following two quantities: 
	\begin{align*}
	  p_c & = \inf\{p\in(0,1):\phi_p^0(0\leftrightarrow\infty)>0\},\\
	  \tilde p_c & = 
	  \sup \Big\{ p \in(0,1) : \lim_{n\rightarrow \infty}-\tfrac1n\log[\phi_p^0(0\longleftrightarrow\partial\Lambda_n)]>0\Big\}.
	\end{align*}
	The claim of the theorem is that $p_c=\tilde p_c$. 
	Obviously, $p_c\ge \tilde p_c$ and we only need to prove the reverse inequality.

	We proceed by contradiction and assume $p_c > \tilde p_c$.
	Then there exist $\tilde p_c < p_0 < p_1 < p_2 < p_c$.
	Corollary~\ref{cor:RSW} implies that $\phi_{p_0}(\calC_h(2n,n))$ is bounded away from $0$, uniformly in $n$. 
	But since $p_1 < p_c$, $\phi_{p_1}(0 \lra \pd \Ball_n) \to  0$ as $n \to \infty$, 
	and Corollary~\ref{cor:gg_applied} yields
	$$ \phi_{p_1}\big(\calC_h(2n,n) \big) \xrightarrow[n \to \infty]{} 1.$$
	In the dual model that translates to  $\phi_{p_1}(\om^* \in \calC_v(2n,n)) \to 0$.
	By Proposition~\ref{prop:easy} applied to the dual random-cluster measure, there exists $c > 0$ such that 
	$\phi_{p_2}(u\stackrel{*}{\longleftrightarrow}v)\le \exp(-c|u-v|)$ for all $u,v \in {\Lat^*}$.
	By Proposition~\ref{prop:main} this contradicts $p_2 < p_c$.
\end{proof}

\begin{proof}[Proposition~\ref{prop:main}]
Let $p$, $\phi_p$ and $c>0$ be as in the proposition. 
For $v\in \calG^*$, let $A(v)$ be the event that there exists a dual-open circuit on $\calG^*$ 
(i.e. a path of dual-open edges of $\calG^*$ starting and ending at the same vertex of $\calG^*$) 
passing through $v$ and surrounding the origin. 
Such a circuit has radius at least $|v|$ when regarded as part of the dual cluster of $v$.
Thus, if $A(v)$ occurs, there exists a vertex $u$ such that 
$u\stackrel{*}{\longleftrightarrow}v$ and $|v|\le |u-v|\le |v|+1$. 
Since $\calG^*$ is locally-finite and doubly-periodic, there exists a constant $C=C(\calG^*)<\infty$ 
not depending on $v$ such that the number of possible vertices $u$ is bounded by $C|v|$. 
A trivial union bound and~\eqref{eq:dual_exp_decay} imply that
$$\phi_{p}(A(v))\le C|v|\exp(-c|v|).$$
The Borel-Cantelli lemma implies that there are almost surely only finitely many $v$ such that $A(v)$ holds,
and therefore finitely many dual-open circuits in $\calG^*$ surrounding the origin. 
This implies that $\phi_{p}(0\longleftrightarrow\infty)>0$ and therefore $p\ge p_c$. 
\end{proof}

\section{Discussion of a possible extension}\label{sec:extensions}\label{sec:conditions}

  The arguments we use in the proof of Theorem~\ref{thm:main} 
  are based on certain specific properties of the model. 
  In addition to the symmetries mentioned explicitly in Theorem~\ref{thm:main}, these are:
  \begin{enumerate}[noitemsep,nolistsep]
  \item positive association (i.e. the FKG inequality), the comparison between boundary conditions and the stochastic ordering;
  \item the domain Markov property~\eqref{domain Markov};
  \item the differential inequalities of Theorems~\ref{thm:ggsh1} and~\ref{thm:hamming}.
  \end{enumerate} 
  One may hope to adapt the result and its proof to other models with these, or similar, properties. 
  We discuss these three conditions next.  
  For illustration consider a family of measures $\mu_p$ on configurations on edges, 
  indexed by some parameter $p \in [0,1]$ called the edge-weight 
  (alternatively they could be parametrized by an inverse temperature $\beta \geq 0$, 
  as in Theorem \ref{thm:main_inhom}). 

  The first condition is classical and also paramount for our approach, 
  we could not hope to proceed without it. 
  
  The second, also fairly classical, is necessary to prove 
  the existence of infinite-volume random-cluster measures, and hence of a critical point. 
  But in this paper it is essentially only used in the proofs
  of Lemma~\ref{lemma:exp_decay} and Proposition~\ref{prop:ggcoupling}.
  One may hope to modify these arguments so as to replace the domain Markov property 
  by alternative properties. We do not have clear candidates. 

  The last condition is more particular and may seem specific to the random-cluster model.  
  Nevertheless, both Theorems~\ref{thm:ggsh1} and~\ref{thm:hamming} follow from rather general arguments. 
  A main ingredient for such inequalities is the existence of a ``Russo-type'' formula of the form
  $$\frac{d}{dp}\mu_p(A)\asymp \mu_p(\mathbf 1_A\eta)-\mu_p(A)\mu_p(\eta)$$
  for any increasing event $A$, where $\eta$ is the number of open edges 
  and $\asymp$ means that the ratio of the two quantities is bounded away from 0 and 1 uniformly in $A$.
  As observed in \cite{Gri06}, measures of the form
  $$\mu_p(\omega)=\frac{1}{Z_p}p^{o(\omega)}(1-p)^{c(\omega)}\mu(\omega),$$
  with $\mu$ a strictly positive measure satisfying the FKG inequality do satisfy the first and third conditions.

\paragraph*{Acknowledgments.} 
The authors are grateful to G. Grimmett for numerous helpful comments on an earlier version of this paper. 
We especially thank him for a suggestion that allowed to simplify substantially 
the third step of the proof of the main theorem.

Both authors were supported by the ERC AG CONFRA, the NCCR SwissMap, as well as by the Swiss {FNS}.
\bibliographystyle{siam} 
\bibliography{biblicomplete}

\footnotesize\obeylines
  \textsc{Universit\'e de Gen\`eve}
  \textsc{Gen\`eve, Switzerland}
  \textsc{E-mail:} \texttt{hugo.duminil@unige.ch; ioan.manolescu@unige.ch}

\end{document}